\documentclass[10pt]{amsart}
\title[]{
Extrinsic diameter of immersed flat tori \\
 in the $3$-sphere I\!I}

\usepackage[dvips]{graphicx} 
\usepackage[active]{srcltx}

\usepackage[mathscr]{eucal}
\usepackage{mathrsfs}
\usepackage{amscd}
\usepackage{amsfonts}
\usepackage{amsmath,amsthm,amssymb}
\usepackage{latexsym}
\usepackage{bm}
\usepackage[dvips]{graphics}
\numberwithin{equation}{section}
\usepackage{txfonts}

\usepackage{amsthm}
\theoremstyle{plain}
 \newtheorem{theorem}{Theorem}[section]
 \newtheorem*{theorem*}{Theorem}

 \newtheorem*{lemma*}{Lemma}
 \newtheorem{proposition}[theorem]{Proposition}
 \newtheorem{fact}[theorem]{Fact}
 \newtheorem*{fact*}{Fact}

 \newtheorem{lemma}[theorem]{Lemma}
 \newtheorem{corollary}[theorem]{Corollary}
\theoremstyle{remark}
 \newtheorem{definition}[theorem]{Definition}
 \newtheorem{remark}[theorem]{Remark}
 \newtheorem*{remark*}{Remark}

\numberwithin{equation}{section}

\newcommand{\R}{{\Bbb R}}
\newcommand{\Z}{{\Bbb Z}}
\newcommand{\braket}[1]{\langle{#1}\rangle}

\renewcommand{\phi}{\varphi}
\renewcommand{\epsilon}{\varepsilon}

\newcommand{\op}{\operatorname}

\newcommand{\su}{\mathfrak s\mathfrak u}

\newcommand{\G}{{\gamma_1,\gamma_2}}
\newcommand{\GB}{{\bar\gamma_1,\bar \gamma_2}}
\newcommand{\g}{\gamma}
\newcommand{\gb}{\bar{\gamma}}
\newcommand{\cb}{\bar{c}}

\author{Kazuyuki~Enomoto}
\address[Enomoto]{%
Tokyo University of Science, Kagurazaka, Shinjuku-ku, 
Tokyo 162-8601, Japan
}
\email{enomoto\_kazuyuki@rs.tus.ac.jp}

\author{Yoshihisa~Kitagawa}
\address[Kitagawa]{%
Department of Mathematics, Utsunomiya University,
Mine-machi, Utsunomiya 321-8505, Japan
}
\email{kitagawa@cc.utsunomiya-u.ac.jp}

\author{Masaaki Umehara}%{Masaaki Umehara}
\address[Umehara]{%
Department of Mathematical and Computing Sciences,
Tokyo Institute of Technology
2-12-1-W8-34, O-okayama Meguro-ku
Tokyo 152-8552 Japan
}
\email{umehara@is.titech.ac.jp}

\subjclass[2000]{Primary 
53C42, %Immersions (minimal, prescribed curvature, tight, etc.) 
\,\, Secondary 53C40.%, Global submanifolds 
}

\thanks{
The third author was partly supported by the Grant-in-Aid for 
Scientific Research (A) 262457005, 
Japan Society for the Promotion of Science.
}

\begin{document}
\begin{abstract}
We show that the extrinsic diameter of
immersed flat tori in the $3$-sphere
is $\pi$ under a certain topological
condition for the projection of
their asymptotic curves with respect to
the Hopf fibration.
\end{abstract}
\maketitle
%\tableofcontents

%%%%%%%%%%%%%%%%%%%%%%%%%%%%%%%%%%%%%%%%%%%%%%%%%%%%%
\section{Introduction}\label{sec1}
Let $S^3$ be the unit sphere in the Euclidean $4$-space
$\R^4$. An immersed surface in $S^3$ is called {\it flat} if
its Gaussian curvature vanishes identically. 
In \cite{EKW} and \cite{KU}, the following problem was posed:

\medskip
\noindent
{\bf Problem.}
{\it Is the extrinsic diameter of an immersed 
flat torus in $S^3$ equal to $\pi$?}

\medskip
The problem is affirmative under the assumptions that
\begin{itemize}
\item[(a)] $f$ is an embedding (cf. \cite{EKW}), or
\item[(b)] the mean curvature function of $f$ does not
change sign (cf. \cite{KU}).
\end{itemize}
If this problem is solved affirmatively,
the rigidity of Clifford tori follows, as pointed out in 
\cite{EKW}. (Weiner \cite{Wei1} also noticed this reduction.)
The references \cite{DS, Wei2, Wei3} are
also related to the problem. 
To investigate global properties of immersed 
flat tori, a representation formula 
given by the second author \cite{Ki1}
is useful:
In this paper, a closed regular curve
means a smooth map
$
\gamma:\R\to S^2 
$
of period $l(>0)$, that is,
\begin{equation}\label{eq:01}
\gamma(s+l)=\gamma(s)\qquad (s\in \R).
\end{equation}
We may think of $\gamma$ as a map 
defined on $\R/l\Z$.
Consider two closed regular curves 
$\gamma_i:\R/l_i\Z\to S^2$ ($i=1,2$).
We denote by $\kappa_i(s)$ the 
geodesic curvature of $\gamma_i$
at $s\in \R$. 
We then fix a positive number $\mu$.
A pair 
$(\gamma_1,\gamma_2)$ of closed regular curves
is said to be {\it $\mu$-admissible}
if the geodesic curvature
of $\gamma_1$ is greater than the maximum of
that of $\gamma_2$,
that is, 
\begin{equation}\label{eq:k1k2}
\min_{t\in \R} \kappa_1(t)
> \mu >  \max_{t\in \R} \kappa_2(t).
\end{equation}
The second author showed in  \cite{Ki1}
that all immersed flat tori can be constructed
from $\mu$-admissible pairs (see Section \ref{sec2}
for details).
If $f_{\gamma_1,\gamma_2}$ is the
immersed flat tori associated to the
pair $(\gamma_1,\gamma_2)$, then 
each $\gamma_i$ ($i=1,2$) can be interpreted
as the projection of a certain asymptotic curve 
of $f_{\gamma_1,\gamma_2}$ via 
the Hopf fibration $p:S^3\to S^2$
(see Remark \ref{rmk:asymptotic} 
in Section~\ref{sec2}). 
We then solve the problem affirmatively 
under a certain topological condition 
for the projections of
the asymptotic curves with respect to
the Hopf fibration,
instead of geometric condition (b)
as in the previous paper \cite{KU}.

To consider the problem, we may assume that
$\gamma_1$ and $\gamma_2$ admit only 
transversal double points (called {\it crossings}), 
without loss of generality
(cf. Proposition \ref{prop:generic2}).
We denote by $\#(\gamma_i)$ ($i=1,2$)
the number of crossings of $\gamma_i$.
If either $\#(\gamma_1)$ or $\#(\gamma_2)$ is
an even number,
then the extrinsic diameter of the immersed tori
associated with $(\gamma_1,\gamma_2)$
is equal to $\pi$ (see Fact \ref{fact:I1}).
So we may assume that $\#(\gamma_1)$ and $\#(\gamma_2)$
are odd integers.
Thus the first non-trivial case
happens when $\#(\gamma_1)=\#(\gamma_2)=1$.
The topological type
of such $\gamma_i$ ($i=1,2$)
are uniquely determined (that is, 
the image of $\gamma_i$ ($i=1,2$)
is homeomorphic to that of the 
eight-figure curve).
Let $p:S^2\setminus \{(0,0,1)\}\to \R^2$
be the stereographic projection.
Then we can draw such a pair 
as plane curves via this projection,
as in Figure \ref{Fig:crossing1}.

\begin{figure}[h]
\begin{center}
        \includegraphics[height=1.8cm]{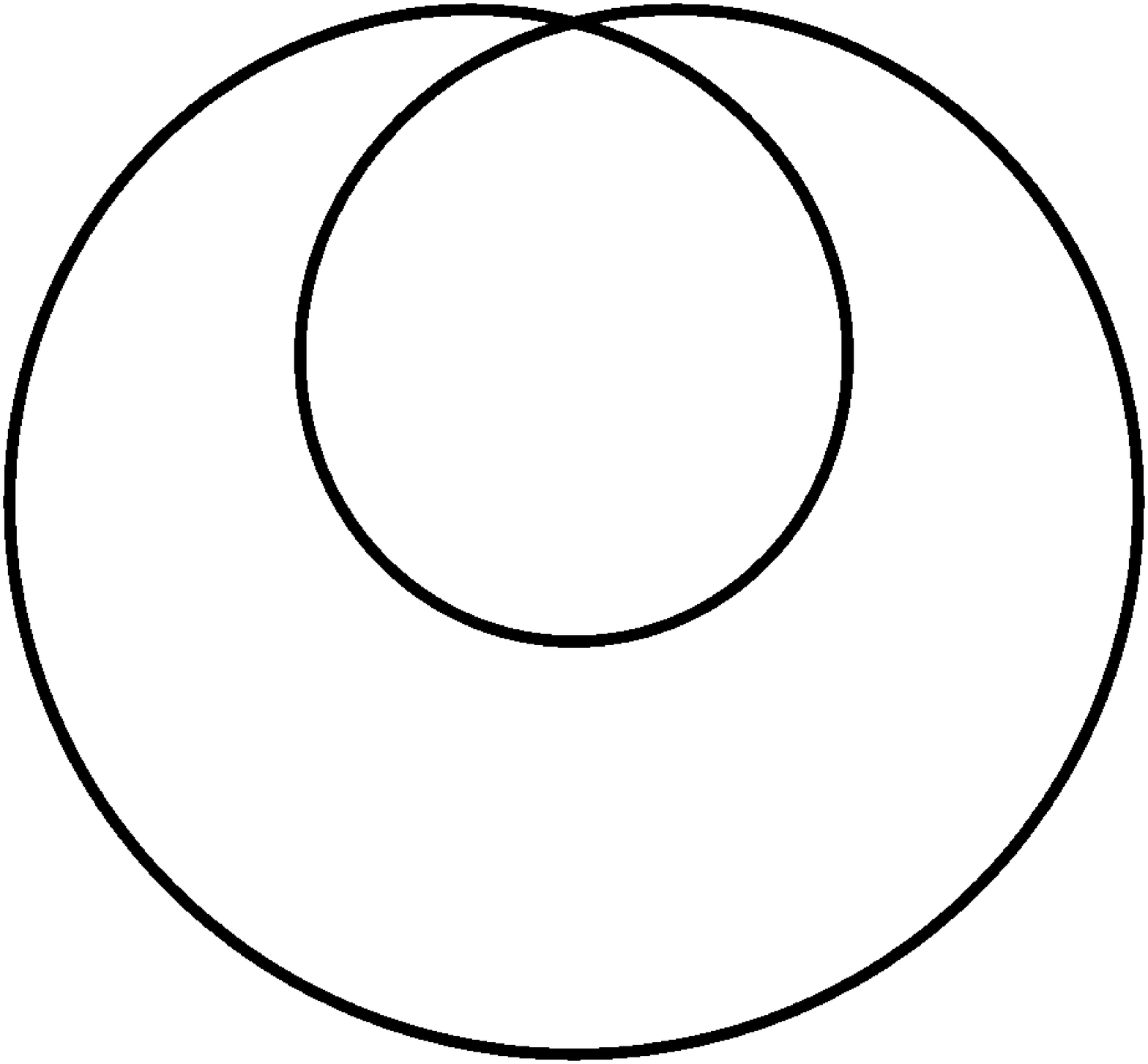}\qquad \qquad
        \includegraphics[height=0.9cm]{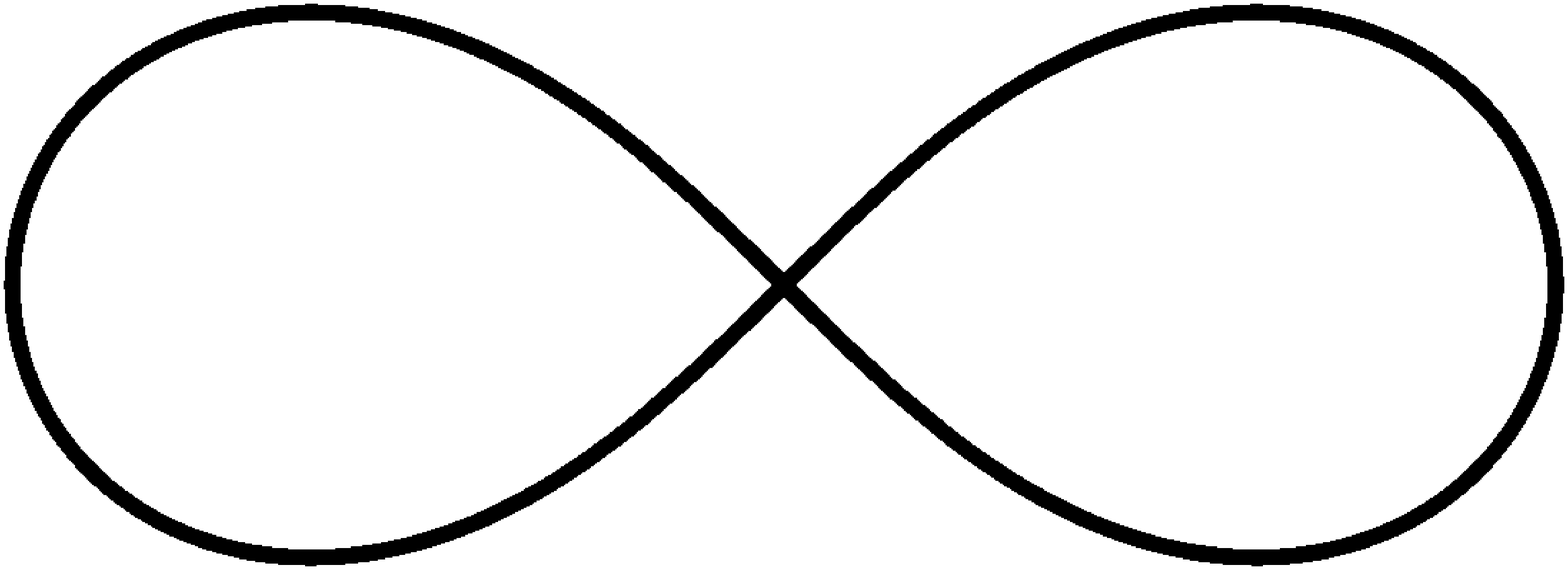}

\caption{A pair of closed curves with $\#(\gamma_1)=\#(\gamma_2)=1$}
\label{Fig:crossing1}
\end{center}
\end{figure}

Unfortunately, even in this particular case,
the techniques in the previous paper \cite{KU}
do not seem to be sufficient.
One of the important results of this paper is 
the following assertion:

\medskip
\noindent
{\bf Proposition A.}\
{\it Suppose that $(\gamma_1,\gamma_2)$ is 
a $\mu$-admissible pair
so that $\#(\gamma_1)=\#(\gamma_2)=1$
$($as in Figure~1$)$.
Then the immersed flat torus associated
to the pair $(\gamma_1,\gamma_2)$
has extrinsic diameter $\pi$.}

\begin{figure}[h!]
\begin{center}
        \includegraphics[height=3cm]{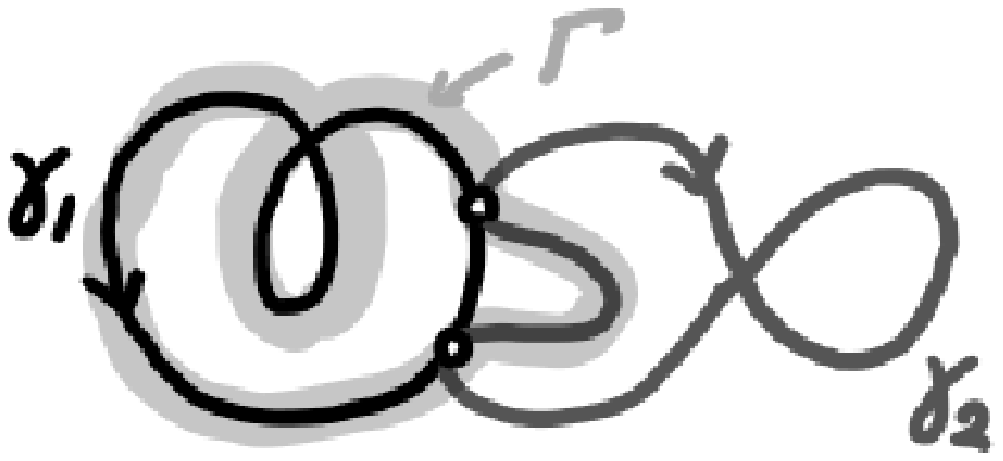}
        \includegraphics[height=3cm]{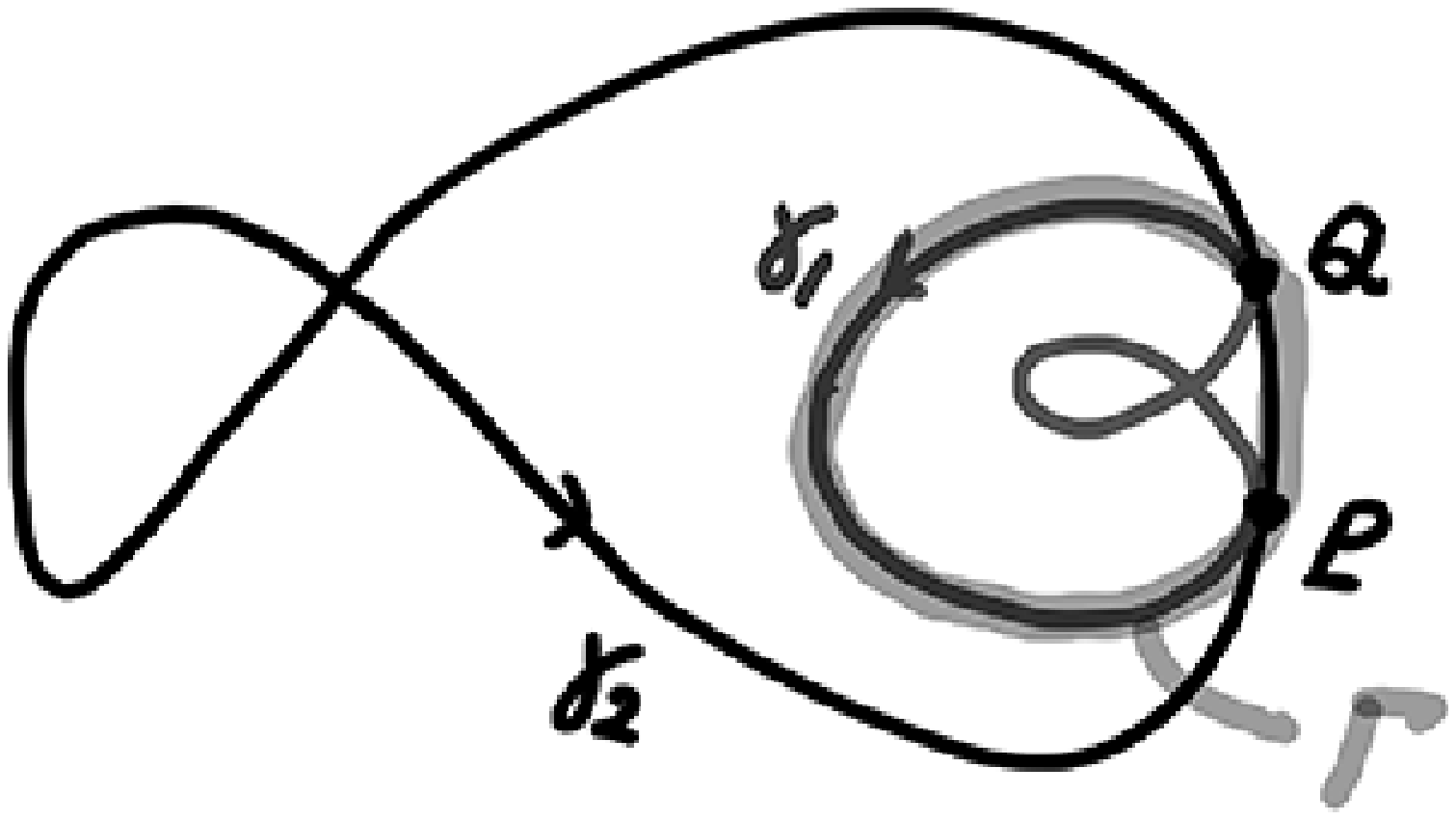}
\caption{Admissible bi-tangents of the first kind (left) and 
the second kind (right), respectively 
}\label{Fig:bi-tangents0}
\end{center}
\end{figure}

This assertion is a special case 
of our main theorem (cf. Theorem B in Section \ref{sec4})
on extrinsic diameters of flat tori,
which is proved by showing
the existence of an admissible 
bi-tangent $(P,Q)$ of the second kind 
(cf. Definition~\ref{def:bi-tg})
between $\gamma_1$ and $\gamma_2$ as in 
Figure~\ref{Fig:bi-tangents0} (right)
under a certain topological assumption. 
We remark that our theorem 
(i.e. Theorem B in Section \ref{sec4})
covers not only the case $\#(\gamma_1)=\#(\gamma_2)=1$
(as in Proposition A),
but also infinitely many topological types of 
$\mu$-admissible pairs.
For example, if $\#(\gamma_1)$ and $\#(\gamma_2)$
are both less than or equal to $3$, then
there are, in total, $10$ topological types 
as possibilities of each $\gamma_i$
(cf.  Figure\ref{Fig:table1} of Appendix B). 
However, there are only two exceptional 
topological types $3_3,3_6$ with 3-crossings.
If $\gamma_1$ and $\gamma_2$
are not of these two types (ignoring
their orientation indicated in Figure \ref{Fig:table1}), 
then the immersed flat torus associated
to the pair $(\gamma_1,\gamma_2)$
has extrinsic diameter $\pi$ (cf. Corollary~\ref{cor:3}).
Such exceptional topological types 
for $\#(\gamma_1),\#(\gamma_2)\le 5$
are also discussed in Section \ref{sec4}.

The paper is organized as follows:
We recall several fundamental properties
of $\mu$-admissible pairs in Section \ref{sec2}.
In Section \ref{sec3}, we prove a key fact
(cf. Proposition  \ref{prop:keyM})
for use in the proof of
our main theorem, and we prove it
in  Section \ref{sec4}. 
At the end of this paper, two appendices are prepared. 
In the first appendix, 
we show that the topological type of 
$\mu$-admissible pairs representing a 
given flat torus is not uniquely determined, in general. 
In the second appendix,
a table (Figure\ref{Fig:table1}) of closed spherical curves within $3$-crossings 
and a table (Figure\ref{Fig:table2}) of closed spherical curves with $5$-crossings
having a special property are given.

\section{$\mu$-admissible pairs of closed curves
}\label{sec2}

\subsection{Kitagawa's representation formula}

Let $\op{SU}(2)$ be the group of all $2 \times 2$ unitary 
matrices with determinant $1$. Its Lie algebra $\su(2)$ 
consists of all $2 \times 2$ skew Hermitian matrices 
of trace zero. 
The adjoint representation of $\op{SU}(2)$ is given by
$
Ad(a)x := axa^{-1}, 
$
where $a \in\op{SU}(2)$ and $x \in \su(2)$.
For $x,y \in \su(2)$, we set 
$$
\braket{x,y} := -\frac12 \op{trace}(xy), \quad
x \times y := \frac12 [x,y], 
$$
where [\ ,\ ] denotes the Lie bracket on $\su(2)$.
Then $\braket{\ ,\ }$  is the canonical $Ad$-invariant 
(positive definite) inner product on $\su(2)$. By setting 
$$
e_1 := \begin{bmatrix}
		0 & \sqrt{-1} \\
		\sqrt{-1} & 0 \\
		\end{bmatrix},\quad
e_2 := \begin{bmatrix}
		0 & -1 \\
		1 & 0 \\
		\end{bmatrix},\quad
e_3 := \begin{bmatrix}
		\sqrt{-1} & 0 \\
		0 & -\sqrt{-1} \\
		\end{bmatrix},
$$
$\{ e_1, e_2, e_3 \}$ gives an orthonormal basis of 
$\su(2)$, and it satisfies the following:
$$
e_1 \times e_2 = e_3, \quad e_2 \times e_3 = e_1, 
\quad e_3 \times e_1 = e_2.
$$
Using this, we can identify $\su(2)$ with 
the Euclidean space $\R^3$.
We endow $\op{SU}(2)$ with a Riemannian metric $\braket{\ ,\ }$ such that 
$\braket{E_i, E_j} = \delta_{ij}$, where $E_i$ denotes the left invariant 
vector field on $\op{SU}(2)$ corresponding to $e_i$. 
Then $\op{SU}(2)$ is isometric to 
the unit 3-sphere $S^3$. Henceforth, we identify 
$S^3$ with $\op{SU}(2)$.

Let $S^2$ be the unit 2-sphere in $\su(2)$ defined by 
$$
S^2 := \{ x \in \su(2) : |x| =1 \}, 
$$
and let $US^2$ be the unit tangent bundle of $S^2$. 
Note that $US^2$ is identified with a subset of 
$\su(2) \times \su(2)$ in the usual way, i.e.,
$$
US^2 := \{(x,y) : |x| = |y| = 1, \braket{x,y} = 0 \},
$$ 
where the canonical projection $p_1 : US^2 \to S^2$ is given by $p_1(x,y) := x$.
Define a map $p_2 : S^3 \to US^2$ by
\begin{equation}\label{eq10:pre}
p_2(a) := (Ad(a)e_3,\ Ad(a)e_1).
\end{equation}
Then the map $p_2$ is a double covering satisfying $p_2(a) = p_2(-a)$.

We consider a regular curve $\gamma:\R\to S^2$.
The {\it geodesic curvature} $\kappa(s)$ of $\gamma$
is defined by
$$
\kappa(s):=\frac{\op{det}(\gamma(s),
\gamma'(s),\gamma''(s))}
{|\gamma'(s)|^3},
$$
where we regard $\gamma(s)\in \R^3$
and $\gamma'(s):=d\gamma/ds$, $\gamma''(s):=d^2\gamma/ds^2$.
The {\it unit normal vector field} of $\gamma(s)$ is given by
$$
n(s) :=\frac{ \g(s)\times \g{'}(s)}{|\g{'}(s)|}.
$$
Then
$$
\kappa(s) = \frac{\braket{\g{''}(s), n(s)}}{|\g{'}(s)|^{2}}
$$
and
\begin{equation}\label{eq15:adm}
n'(s) = -\kappa(s){\g}'(s)
\end{equation}
hold.
We then consider the curve $\hat{\g}(s)$ in $US^2$ defined by
\begin{equation}\label{eq20:pre}
\hat{\g}(s) := \left(\g(s),\ \frac{\g{'}(s)}{|\g{'}(s)|}\right).
\end{equation}
Since $p_2 : S^3 \to US^2$ is a covering, 
there exists a regular curve $c_\gamma(s)$ in $S^3$ such that
\begin{equation}\label{eq22:pre}
p_2\circ c_\gamma(s) = \hat{\g}(s).
\end{equation}
The curve $c_\gamma(s)$ is called a {\it lift} of the curve 
$\hat{\g}(s)$ with respect to the covering $p_2$.
(Such a lift is determined up to a sign, that is,
$-c_\gamma$ is also a lift of $\gamma$.)
It follows that
\begin{equation}\label{eq30:pre}
c_\gamma(s)^{-1}c'_\gamma(s) 
= \frac{|\g{\ '}(s)|}2 \biggl(e_2 + \kappa(s)e_3
\biggr). 
\end{equation}
In particular, $|{\g}'|^2(1 + \kappa^2)$ is
equal to $4$ if and only if $s$ is the arc-length parameter of
$c_\gamma$ with respect to the canonical
Riemannian metric of $S^3$. Regarding this,
we give the following 
definition of \lq admissible pairs\rq.

\begin{definition}[\cite{Ki1}, \cite{KU}] \label{def:PAP}
A pair $(\gamma_1,\gamma_2)$ of regular curves in $S^2$
defined on $\R$ 
is called an {\it admissible pair} 
if it satisfies the following conditions :
\begin{itemize}
\item[(1)]
$|{\g_i}'|^2(1 + \kappa_i^2) = 4 \quad \text{for \ } i=1, 2$,
\item[(2)]
$\kappa_1(s_1) > \kappa_2(s_2) \quad \text{for all \ } s_1, s_2 \in \R$,
\end{itemize}
where $\kappa_i$ denotes the geodesic curvature of $\g_i$. 
\end{definition}

Let $(\g_1, \g_2)$ be an admissible pair, 
and let $c_i : \R \to S^3$ be the lift of the curve $\widehat{\g_i}$ with 
respect to the covering $p_2 : S^3 \to US^2$. Using the group 
structure on $S^3$, we define a map $f_\G : \R^2 \to S^3$ by
\begin{equation}\label{eq3:adm}
f_\G(s_1,s_2) := c_1(0)^{-1}c_1(s_1)c_2(s_2)^{-1}c_2(0).
\end{equation}
Since $p_2:S^3\to US^2$ is a double covering,
there are two possibilities for the lift $c_i$ of $\hat \gamma_i$.
However, the map $f_\G$ is independent of the choices 
of the lifts $c_1$ and $c_2$. 
The first fundamental form $ds^2$ 
and the second fundamental form $h$ of $f_\G$
are given by
\begin{equation}\label{eq:chebyshev}
ds^2 = {ds_1}^2 + 2 \cos \omega {ds_1}{ds_2} + {ds_2}^2,\qquad
h = 2 \sin \omega \,{ds_1}{ds_2},
\end{equation}
respectively, where 
$$
\omega(s_1, s_2):= 
\cot^{-1}\left(\kappa_1(s_1)\right) + \cot^{-1}\left(-\kappa_2(s_2)\right).
$$
Since $f_{\gamma_1,\gamma_2}$ is flat, the Gaussian curvature 
of $f_\G$ vanishes identically.
On the other hand,  the mean curvature function $H_{\G}$ satisfies 
\begin{equation}\label{eq:Mean}
H_{\G}(s_1, s_2) = \frac{1 + \kappa_1(s_1)\kappa_2(s_2)}{\kappa_1(s_1) 
- \kappa_2(s_2)}.
\end{equation}
Using the function
$\alpha_i(s)$ given by 
$
\tan \alpha_i(s) := \kappa_i(s)\,\,(|\alpha_i(s)| < {\pi}/2),
$ 
we obtain
$$
\omega(s_1, s_2) = \pi - \alpha_1(s_1) + \alpha_2(s_2).
$$
The following assertion holds:

\begin{fact}[{\cite[Theorem 4.3]{Ki1}}]\label{fact:converse}
For a complete flat immersion
$f:\R^2\to S^3$ whose mean curvature function is bounded,
there exists an admissible pair $(\gamma_1,\gamma_2)$
such that $f_{\gamma_1,\gamma_2}$ is congruent to $f$ in $S^3$.
\end{fact}

\begin{remark}\label{rmk:asymptotic}
By \eqref{eq:chebyshev}, each $c_i$ ($i=1,2$) is an asymptotic curve
of $f$. Since $p:=p_1\circ p_2:S^3\to S^2$ gives the Hopf fibration,
each $\gamma_i$ can be considered as the projection of an asymptotic
curve of $f$ by $p$. 
\end{remark}

We next give a necessary and sufficient condition
for two given admissible pairs to give the
same immersed flat torus:
Let $(\G)$ and $(\GB)$ be 
admissible pairs.
We denote by $(\G) \equiv (\GB)$ 
the existence 
$\alpha_1, \alpha_2 
\in \op{SO}(3)$ satisfying
$$
\alpha_1(\g_1(s)) = \gb_1(s),\quad
\alpha_2(\g_2(s)) = \gb_2(s).
$$
The following assertion holds:

\begin{proposition}
$(\G) \equiv (\GB)$ if and only if $f_{\G} = f_{\GB}$. 
\end{proposition}

\begin{proof}
Suppose that $(\G) \equiv (\GB)$. 
Since the map $Ad :S^3(=\op{SU}(2)) \to \op{SO}(3)$ is surjective, 
there exist $g_1,g_2 \in S^3$ such that 
$Ad(g_i)\g_i(s) = \gb_i(s)$ ($i=1,2$).
Let $c_i(s)$ be a lift of $\widehat{\g_i}(s)$ with respect to 
the covering $p_2 : S^3 \to US^2$, and let $\cb_i(s)$ 
be a curve in $S^3$ defined by $\cb_i(s) := g_ic_i(s)$. Then we have
$$
\begin{aligned}
Ad(\cb_i(s))e_3 &= Ad(g_i)Ad(c_i(s))e_3 = Ad(g_i)\g_i(s) = \gb_i(s),\\
Ad(\cb_i(s))e_1 &= Ad(g_i)Ad(c_i(s))e_1 = \frac{Ad(g_i)\g_i{'}(s)}{|\g_i{'}(s)|}
= \frac{\gb_i{'}(s)}{|\gb_i{'}(s)|}.
\end{aligned}
$$
This shows that $p_2\circ \cb_i = \widehat{\gb_i}$. So we obtain 
$$
f_{\GB}(s_1,s_2) = \cb_1(0)^{-1}\cb_1(s_1)\cb_2(s_2)^{-1}\cb_2(0) 
= c_1(0)^{-1}c_1(s_1)c_2(s_2)^{-1}c_2(0) = f_{\G}(s_1, s_2).
$$

Conversely, we suppose that $f_{\G} = f_{\GB}$. Let $c_i$ (resp. $\cb_i$) 
be a curve in $S^3$ such that 
$$
p_2\circ c_i = \widehat{\g_i} \qquad
(\mbox{resp.}\quad  p_2\circ \cb_i = \widehat{\gb_i}).
$$ 
Since $f_{\G}(s,0) = f_{\GB}(s,0)$ and $f_{\G}(0,s) = f_{\GB}(0,s)$, we obtain
$\cb_i(s) = g_ic_i(s)$ ($i=1,2$), where $g_i = \cb_i(0)c_i(0)^{-1}$. This shows that
$$
\gb_i(s) = Ad(\cb_i(s))e_3 = Ad(g_i)Ad(c_i(s))e_3 = Ad(g_i)\g_i(s)
\quad (i=1,2).
$$      
Hence $(\G) \equiv (\GB)$ holds.
\end{proof}

We prove the following:

\begin{proposition}\label{prop:T10}
Let $(\g_1,\ \g_2)$ be an  admissible pair, and let  
\begin{equation}\label{eq:g1g2}
\gb_1(s):= -\g_2(s),\quad \gb_2(s):= -\g_1(s).
\end{equation}
Then the pair $(\gb_1,\ \gb_2)$ is an
admissible pair such that 
\begin{equation}\label{eq:k-pm}
\bar{\kappa}_1(s) = -\kappa_2(s),\quad \bar{\kappa}_2(s) 
= -\kappa_1(s),
\end{equation}
and there exists $g \in S^3$ 
such that $f_{\GB}(s_1,s_2) 
= g^{-1}\left(f_{\G}(s_2,s_1)\right)^{-1}g$.
In particular,  $f_{\GB}(s_1,s_2)$ is
congruent to $f_{\G}(s_2,s_1)$ in $S^3$.  
\end{proposition}

\begin{proof}
Since ${\gb_1}'(s) = -{\g_2}'(s)$, 
we obtain $\bar{n}_1(s) = n_2(s)$.
Hence
$$
\bar{\kappa}_1(s) = 
\frac{\braket{{\gb_1}''(s),\  \bar{n}_1(s)}}{|{\gb_1}'(s)|^2} 
= \frac{\braket{-{\g_2}''(s),\ n_2(s)}}{|{\g_2}'(s)|^2} = -\kappa_2(s).
$$
By the same way, we obtain $\bar{\kappa}_2(s) = -\kappa_1(s)$.
Hence, $\bar{\kappa}_1(s_1) > \bar{\kappa}_2(s_2)$, 
and so $(\GB)$ is an admissible pair.
Let $c_i$ be a curve in $S^3$ 
such that $p_2\circ c_i=\widehat{\g_i}$ ($i=1,2$). 
We choose $g \in S^3(=\op{SU}(2))$ such that
$$
Ad(g)e_1 = - e_1,\quad
Ad(g)e_2 = e_2,\quad
Ad(g)e_3 = - e_3,
$$
and set $\bar{c}_1(s) := c_2(s)g,\ \bar{c}_2(s) := c_1(s)g$.
Then we have
\begin{align*}
Ad(\bar{c}_1(s))e_3 &= -Ad(c_2(s))e_3 = -\g_2(s) = \gb_1(s),\\
Ad(\bar{c}_1(s))e_1 &= -Ad(c_2(s))e_1 
= -\frac{{\g_2}'(s)}{|{\g_2}'(s)|} 
= \frac{{\gb}'_1(s)}{|{\gb}'_1(s)|}.
\end{align*}
This shows that $p_2\circ \bar{c}_1(s) = \widehat{\gb_1}(s)$.
By the same way, we obtain $p_2\circ \bar{c}_2(s) = \widehat{\gb_2}(s)$.
Hence
$$
\begin{aligned}
f_{\GB}(s_1, s_2) &= \cb_1(0)^{-1}\cb_1(s_1)\cb_2(s_2)^{-1}\cb_2(0) = 
g^{-1}c_2(0)^{-1}c_2(s_1)c_1(s_2)^{-1}c_1(0)g\\ 
&= g^{-1}\left(c_1(0)^{-1}c_1(s_2)c_2(s_1)^{-1}c_2(0)\right)^{-1}g 
= g^{-1}\left(f_{\G}(s_2, s_1)\right)^{-1}g.
\end{aligned}
$$
This completes the proof.
\end{proof}

\subsection{Periodic admissible pairs
and $\mu$-admissible pairs}

Let $\gamma:\R/l\Z\to S^2$ be a closed regular
curve (cf. \eqref{eq:01}).
In this paper, a {closed interval} $[a,b](\subset \R)$ means 
a proper sub-interval of $S^1:=\R/l\Z$ 
satisfying $|a-b|<l$.

\begin{definition}
An admissible pair 
$(\gamma_1,\gamma_2)$
is called {\it periodic} if 
$\gamma_1$ and $\gamma_2$ 
are both periodic, that is,
there exist two positive numbers $l_i$ ($i=1,2$)
such that $\gamma_i$ has period $l_i$.
\end{definition}

\begin{remark}
We dropped the condition (1) of Definition \ref{def:PAP}
to define periodic admissible 
pairs in the previous paper \cite{KU}.
So the above definition is stronger than
that given in \cite{KU}.
\end{remark}

\begin{proposition}[\cite{Ki1}]
The flat immersion $f_{\G}:\R^2\to S^3$
induced by $(\G)$ gives a flat torus
if and only if 
$(\gamma_1,\gamma_2)$
is a periodic admissible pair.
\end{proposition}

\begin{proof}
Suppose that $f_\G$ gives a torus.
Kitagawa \cite[Theorem A]{Ki1}
showed that asymptotic 
curves on an immersed 
flat torus are all closed curves.
Since $c_1,c_2$ induced by
$\gamma_1,\gamma_2$ are
asymptotic curves of $f_{\gamma_1,\gamma_2}$
by \eqref{eq:chebyshev},
we have that $c_1,c_2$ are closed, 
and so $\gamma_1, \gamma_2$
are periodic by \eqref{eq22:pre}.
Conversely, if $\gamma_1, \gamma_2$
are periodic, 
\eqref{eq20:pre} and
\eqref{eq22:pre} yield that
$c_i(s+l_i)=\pm c_i(s)$ for $i=1,2 $.
In particular, we have 
$c_i(s+2l_i)=c_i(s)$
($i=1,2$), and
then $f_\G$ is doubly periodic
by \eqref{eq3:adm},
proving the assertion.
\end{proof}

If $\gamma_i$ ($i=1,2$) are periodic,
then 
$\min_{s\in \R} \kappa_i(s)$
and $\max_{s\in \R} \kappa_i(s)$
exist. So (2) of
Definition~\ref{def:PAP} yields the 
following:

\begin{proposition}
If $(\gamma_1,\gamma_2)$
is a periodic admissible pair,
then there exists $\mu\in \R$
such that
\begin{equation}\label{eq:k1k200}
\min_{s\in \R} \kappa_1(s) > \mu >  
\max_{s\in \R} \kappa_2(s).
\end{equation}
\end{proposition}

The following assertion holds: 

\begin{proposition}
In the setting as in 
\eqref{eq:k1k200},
we may assume $\mu>0$.
\end{proposition}

\begin{proof}
Let $(\gamma_1,\gamma_2)$ be a 
periodic admissible pair. 
If $\mu=0$, then
we may adjust $\mu$ to be positive
satisfying 
$
\min_{s\in \R} \kappa_1(s)
> \mu >0.
$
If $\mu<0$, then
the $\mu$-admissible pair $(\gb_1,\gb_2)$ given by
\eqref{eq:g1g2} has the property 
(cf. Proposition \ref{prop:T10})
$$
\min_{s\in \R} \bar \kappa_1(s)
> \bar \mu >  \max_{s\in \R} \bar \kappa_2(s),
$$
where $\bar \mu:=-\mu>0$.
By replacing $(\gamma_1,\gamma_2)$
by $(\bar \gamma_1,\bar \gamma_2)$,
we get the conclusion.
\end{proof}

If $(\gamma_1,\gamma_2)$ is a periodic
admissible pair, then its parameter $s$
must satisfy (1) of Definition~\ref{def:PAP}.
To remove this, we prepare the following terminology:

\begin{definition}
A pair of closed regular curves $(\gamma_1,\gamma_2)$
is said to be {\it $\mu$-admissible} if
it satisfies \eqref{eq:k1k200}
for positive $\mu$.
\end{definition}

By definition, a periodic admissible pair 
satisfying 
\eqref{eq:k1k200} with $\mu>0$
is 
$\mu$-admissible. Conversely 
the following assertion holds:

\begin{proposition}
Let $(\gamma_1,\gamma_2)$ be a
$\mu$-admissible pair. Then after 
suitable changes
of parameters of $\gamma_i$ ($i=1,2$),
$(\gamma_1,\gamma_2)$ gives a
periodic admissible pair.
\end{proposition} 

\begin{proof}
Suppose that the parameter $t$ of $\gamma_i(t)$ ($i=1,2$)
is general, perhaps not satisfying (1) of 
Definition \ref{def:PAP}.
Since $\gamma_i$ is periodic, there exists $d_i(>0)$
such that
$$
\gamma_i(t+d_i)=\gamma_i(t)\qquad (t\in \R).
$$
We set
$$
s_i(t):=\frac12\int_0^t \left|\frac{{\g_i}(u)}{du}\right|
\sqrt{1 + \kappa_i(u)^2} \,du
\qquad (i=1, 2).
$$
Since $s_i:\R\to \R$
is a bijection satisfying $ds_i/dt>0$,
we can choose $s_i$ as a  new parameter of $\gamma_i$.
Then $|\gamma'_i(s_i)|\sqrt{1 + \kappa_i(s_i)^2}=2$ and
 $\gamma_i(s_i+l_i)=\gamma(s_i)$ hold,
where $l_i:=s_i(d_i)$ ($i=1,2$).
\end{proof}

\subsection{The regular homotopical invariant $I$ 
of closed spherical curves}

For a closed regular curve $\gamma(t)$ in $S^2$
of period $l$, we have shown the existence of a 
new parameter $s$ so that 
$$
\gamma(s+l)=\gamma(s)\qquad (s\in \R),\qquad
|\gamma'(s)|\sqrt{1 + \kappa(s)^2}=2,
$$
as seen above.
Then we set
\begin{equation}
I(\gamma):=\begin{cases}
0 & \mbox{if $c_\gamma(s+l)=c_\gamma(s)$}, \\
1 & \mbox{if $c_\gamma(s+l)=-c_\gamma(s)$}. 
\end{cases}
\end{equation}

\begin{definition}\label{def:generic}
A closed regular curve
$\gamma$ is said to be {\it generic} if all of
its self-intersections are
transversal double points.
Moreover, a $\mu$-admissible pair $(\gamma_1,\gamma_2)$
is said to be {\it generic} if each $\gamma_i$ ($i=1,2$) is 
generic.
\end{definition}

If $\gamma$ is generic, the number of crossings
(i.e. self-intersections) is finite, and we denote it by
$\#(\gamma)$. Then it is 
well-known that 
\begin{equation}
I(\gamma)=\begin{cases}
0 & \mbox{if $\#(\gamma)$ is odd}, \\
1 & \mbox{if $\#(\gamma)$ is even}. 
\end{cases}
\end{equation}
Suppose that $(\gamma_1,\gamma_2)$
is a periodic admissible pair, and
each $\gamma_i$ has the period $l_i(>0)$,
and $I(\gamma_1)=1$.
Then
\begin{align*}
f_{\G}(s_1+l_1,s_2)
&=c_1(0)^{-1}c_1(s_1+l_1)c_2(s_2)^{-1}c_2(0)^{-1}\\
&=c_1(0)^{-1}(-c_1(s_1))c_2(s_2)^{-1}c_2(0)^{-1}
=-f_{\G}(s_1,s_2).
\end{align*}
Similarly, $I(\gamma_2)=1$ also implies that
$
f_{\G}(s_1,s_2+l_2)
=-f_{\G}(s_1,s_2).
$
Consequently, we get the following:

\begin{fact}\label{fact:I1}
The extrinsic diameter of 
$f_{\gamma_1,\gamma_2}$ is equal to $\pi$ if $I(\gamma_1)=1$ 
or $I(\gamma_2)=1$.
\end{fact}

Thus, to prove Proposition A, we may assume that 
$I(\gamma_1)=I(\gamma_2)=0$.

\begin{figure}[h]
\begin{center}
        \includegraphics[height=2.3cm]{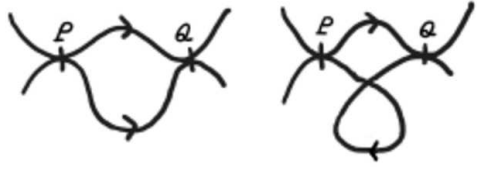}
\caption{Admissible bi-tangents of the first kind (left)
and second kind (right)}\label{Fig:bi-tangents}
\end{center}
\end{figure}

Let $\gamma_i:\R/l_i\Z\to S^2$ ($i=1,2$)
be two closed regular curves of period $l_i(>0)$
in $S^2$ with
unit normal vector fields $n_i$ along $\gamma_i$.
Let
$a_i, b_i$ 
$(0<a_i < b_i<l_i)$ 
be real numbers satisfying 
\begin{equation}\label{eq:AB}
(P:=)\gamma_1(a_1)=\gamma_2(a_2),\quad
(Q:=)\gamma_1(b_1)=\gamma_2(b_2).
\end{equation}
If it holds that
\begin{equation}\label{eq:plus}
n_1(a_1)=n_2(a_2),\quad
n_1(b_1)=n_2(b_2)
\end{equation}
or 
\begin{equation}\label{eq:minus}
n_1(a_1)=n_2(a_2),\quad
n_1(b_1)=-n_2(b_2),
\end{equation}
then the pair of points $(P,Q)$ on $S^2$ is called
a {\it bi-tangent} between $\gamma_1$ and $\gamma_2$.
Moreover, if the former case (i.e.~\eqref{eq:plus})
happens, then the bi-tangent is said to be
 {\it admissible}.

\begin{definition}\label{def:bi-tg}
We set
$
\check \gamma_i:=(\gamma_i, n_i):\R\to US^2 \,\, (i=1,2).
$
An admissible  bi-tangent $(P,Q)$ is called of {\it  the first kind}
if the restriction $\check \gamma_i|_{[a_1,b_1]}$ 
on the interval $[a_1,b_1]$
is regular homotopic to
$\check\gamma_2|_{[a_2,b_2]}$ 
keeping the properties \eqref{eq:AB}
and \eqref{eq:plus}.
If $(P,Q)$ is not of the first kind, it is called  of {\it the
second kind} (see Figure~\ref{Fig:bi-tangents} right).
\end{definition}

The following assertion was proved in \cite[Proposition 2.3]{KU}:
 
\begin{fact}\label{prop:2red}
Let $(\gamma_1,\gamma_2)$ be a 
$\mu$-admissible pair.
Then the following two assertions are equivalent.
\begin{enumerate}
\item The flat immersed torus $f_{\gamma_1,\gamma_2}$
associated to $(\gamma_1,\gamma_2)$
has extrinsic diameter $\pi$.
\item 
There exists an orientation preserving isometry 
$\phi$ on $S^2$ such that a subarc of $\gamma_1$ 
and a subarc of $\phi\circ \gamma_2$
have an admissible bi-tangent of the second kind.
\end{enumerate}
\end{fact}

We have the following reduction for the problem given
 in the introduction.

\begin{proposition}\label{prop:generic2}
The problem in the introduction is solved affirmatively if
there exists an admissible bi-tangent of the second kind 
for each generic $\mu$-admissible pair. 
\end{proposition}

\begin{proof}
This assertion follows from the fact that 
a non-generic $\mu$-admissible pair
can be obtained by the limit of a sequence of generic pairs.
So we obtain the conclusion regarding the 
fact that $f_{\gamma_1,\gamma_2}$
depends on $\gamma_1,\gamma_2$
continuously.
\end{proof}

\subsection{A tool to control the behavior of $\gamma_1$}

\begin{figure}[htb]
\begin{center}
        \includegraphics[height=4.5cm]{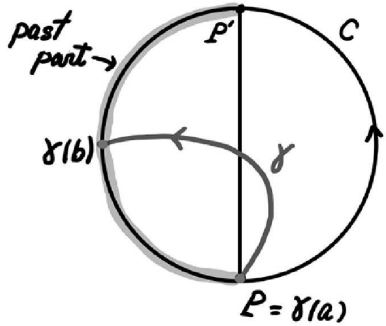}
\caption{Behavior of $\gamma$ starting from $C$ into the 
interior domain}
\label{Fig:ku-lemma}
\end{center}
\end{figure}

Let $C$ be an oriented circle in $S^2$ of positive geodesic curvature. 
Then $C$ separates $S^2$ into 
two open regions. The smaller one is called the {\it interior domain}
and denote it by $\Delta_C$. 
Let $O$ be the center of $C$ in the interior domain of $C$ and 
$P$ a point on $C$. Then the great circle passing through $P$ 
and $O$ meets $C$ at the antipodal point $P'$
with respect to the circle $C$.
Let $[P,P']$ (resp. $[P',P]$) be the subarc of $C$ from $P$ to $P'$
(resp. $P'$ to $P$). Then $C$ is a union of two arcs $[P,P']$ and $[P',P]$.
We call $[P,P']$ (resp. $[P',P]$) the {\it future part}
(resp. the {\it past part}) of the oriented circle $C$ 
with respect to P. The following assertion was proved in
\cite{KU}, which is a useful tool for investigating the behavior of 
the curve $\gamma$ whose geodesic curvature is greater than $\mu(>0)$.

\begin{fact}[{\cite[Lemma 4.5]{KU}}]\label{fact:KU1}
Let $C$ be an oriented circle of positive constant geodesic 
curvature $\mu(> 0)$ centered at a point $O$ in $S^2$. 
Let $\gamma:[a,b]\to S^2$  be a simple regular arc
(i.e. an arc without self-intersections) 
whose geodesic
curvature is greater than $\mu$ for all $t\in [a,b]$. 
Suppose that $\gamma((a,b))$ lies in $\Delta_C$,
and the two endpoints $\gamma(a),\,\, \gamma(b)$ both lie 
on the circle $C$. Then $\gamma(b)$ lies in the past part of 
$C$ with respect to $\gamma(a)$
$($see Fig \ref{Fig:ku-lemma}$)$.
\end{fact}

\begin{figure}[h]
\begin{center}
        \includegraphics[height=2.3cm]{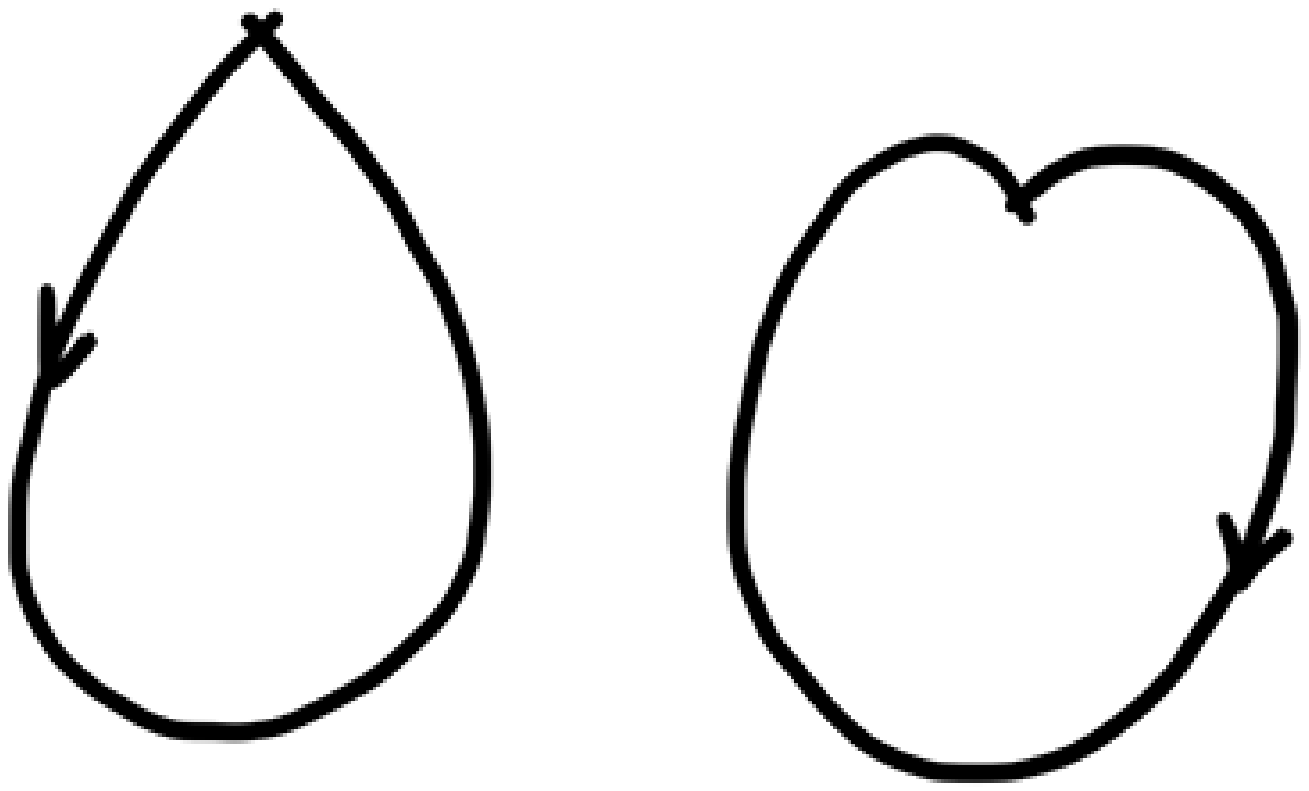}\quad \qquad
        \includegraphics[height=2.0cm]{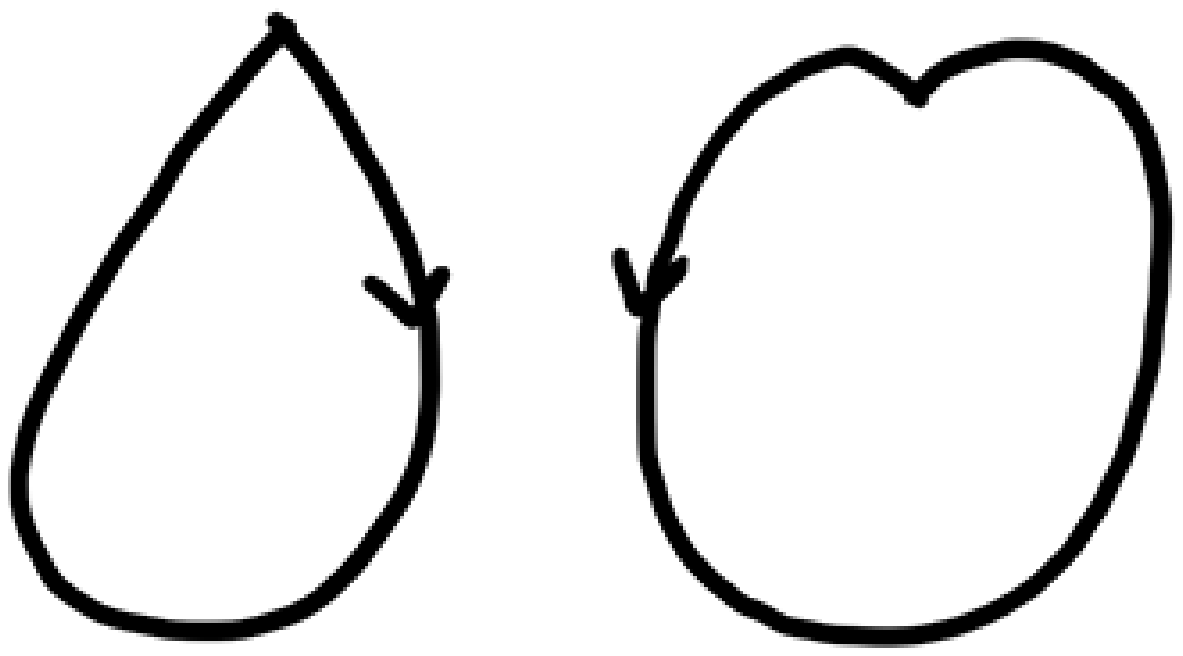}
\caption{Two positive shells (left)
and two negative shells (right)}\label{Fig:shells}
\end{center}
\end{figure}

\section{Bi-tangent $\mu$-circles inscribed in a shell}
\label{sec3}

We give here the concept of \lq shell\rq,
which is used throughout this paper:
 
\begin{definition}
A regular arc $\gamma:[a,b]\to S^2$ is 
called a {\it shell} if 
\begin{itemize}
\item $\gamma([a,b))$ 
has no self-intersection,
\item $\gamma(a)=\gamma(b)$, and 
\item  $\gamma'(a)$ and $\gamma'(b)$ are linearly independent.
\end{itemize}
The point $\gamma(a)=\gamma(b)$ is called the {\it node} 
of the shell.
Moreover, the shell is called {\it positive}
(resp. {\it negative}) if the determinant function satisfies
$$
\det(\gamma(a),\gamma'(a),\gamma'(b))<0
\qquad (\mbox{resp.} \,\, \det(\gamma(a),\gamma'(a),\gamma'(b))>0).
$$
\end{definition}

Since the image of the shell $\gamma$ is a simple closed curve,
it separates $S^2$ into  two open regions. The domain whose interior 
angle at the node is less than $\pi$ 
is called the {\it interior domain} of the shell
and is denoted it by $\Delta_\gamma$. 
Then the  shell $\gamma:[a,b]\to S^2$ 
is  positive (resp. negative), 
if $\Delta_\gamma$ lies in the left-hand side 
(resp. the right-hand side) of $\gamma$.

The classical four vertex theorem assets that
there are four vertex (i.e.  four critical points 
of the curvature function)
on a given  simple closed curve in $\R^2$.
By the stereographic projection,
circles in $S^2$ are corresponding to
circles or lines in $\R^2$.
In particular, the critical points 
of the geodesic curvature function  
(resp. the osculating circles)
of a given spherical regular curve
corresponds to the vertices (resp. the osculating circles)
 of the corresponding plane curve.
An {\it osculating circle} at a point on a curve
is a circle which has second order 
contact with the curve at the point.
As a refinement of the four vertex theorem, 
Kneser \cite{Kn} proved that there are 
two distinct inscribed osculating
circles and two distinct circumscribed osculating
circles of a simple closed spherical curve. 
As an analog of this, the following fact is 
known (cf. \cite[Proposition B.1]{KU}):

\begin{fact}[Strong version of Jackson's lemma]
\label{prop:pnshell}
Suppose that a closed regular spherical
curve  $\gamma$ has a shell on a closed
interval $[a,b](\subset S^1)$.
Then there exists $c\in (a,b)$
such that the osculating circle $C$ of $\gamma$
at $\gamma(c)$ lies in the closure of the
interior domain of the shell $\gamma([a,b])$.
\end{fact}

Jackson \cite{J} proved that 
the geodesic curvature function attains 
at least one local maximum on a
given positive shell. The above fact
is a refinement of this assertion.

Throughout this section, we
fix a $\mu$-admissible pair $(\gamma_1,\gamma_2)$
such that $\gamma_2$ has a positive shell,
that is, the restriction of  $\gamma_2$ to
a closed interval
$[a_2,b_2](\subset \R/l_2\Z)$ gives
a positive shell.
We now fix this to be so, and set 
\begin{equation}\label{eq:Sigma2}
\Sigma_2:=\gamma_2([a_2,b_2]),
\end{equation}
which is the image of the positive shell.

\begin{definition}
A circle on $S^2$ is called 
a {\it $\mu$-circle} if its geodesic
curvature is $\mu(>0)$.
\end{definition}

\begin{lemma}\label{lem:Sigma2}
There exists a 
$\mu$-circle which 
is inscribed in $\Sigma_2$ 
(i.e. the $\mu$-circle lies in 
the closure of the interior domain of the positive
shell) and
meets 
$\Sigma_2$ at more than one point.
\end{lemma}

We call such a $\mu$-circle 
\em{a bi-tangent circle}
of $\Sigma_2$.

\begin{proof}
By Fact \ref{prop:pnshell},
we can find an inscribed osculating circle
$\Gamma$ which meets $\Sigma_2$ at 
$$
P_0=\gamma_2(t_0)\in \Sigma_2\qquad (t_0\in (a_2,b_2)).
$$
We let $\Gamma_1$ be the $\mu$-circle
which is tangent to $\Sigma_2$ at $P_0$. 
Since $\mu$ is greater than the maximum of 
the geodesic curvature of $\gamma_2$,
$\Gamma_1$ lies in the left-hand side of $\Gamma$. 
In particular, 
$\Gamma_1$ is a $\mu$-circle inscribed in $\Sigma_2$.
For each $t\in(a_2,t_0]$, we let $C_t$
be the $\mu$-circle which is tangent to $\Sigma_2$
at $\gamma_2(t)$.
Then $\Gamma_1=C_{t_0}$ holds. We set
$$
t_1:=\inf\biggl\{
t\in (a_2,t_0]\,;\, \mbox{$C_t$ is inscribed in $\Sigma_2$.}\biggr\}
$$
Since $\mu$ is greater than the maximum of
the geodesic curvature of $\gamma_2$,
the circle 
$C_{t_1}$ meets $\Sigma_2$ somewhere other than $\gamma(t_1)$.
(Otherwise, we can roll $C_{t_1}$ towards the 
node of the shell, and then we can find $C_t$ ($t<t_1$)
which is inscribed in $\Sigma_2$.) 
So we get the assertion.
\end{proof}

\begin{figure}[h]
\begin{center}
        \includegraphics[height=3.0cm]{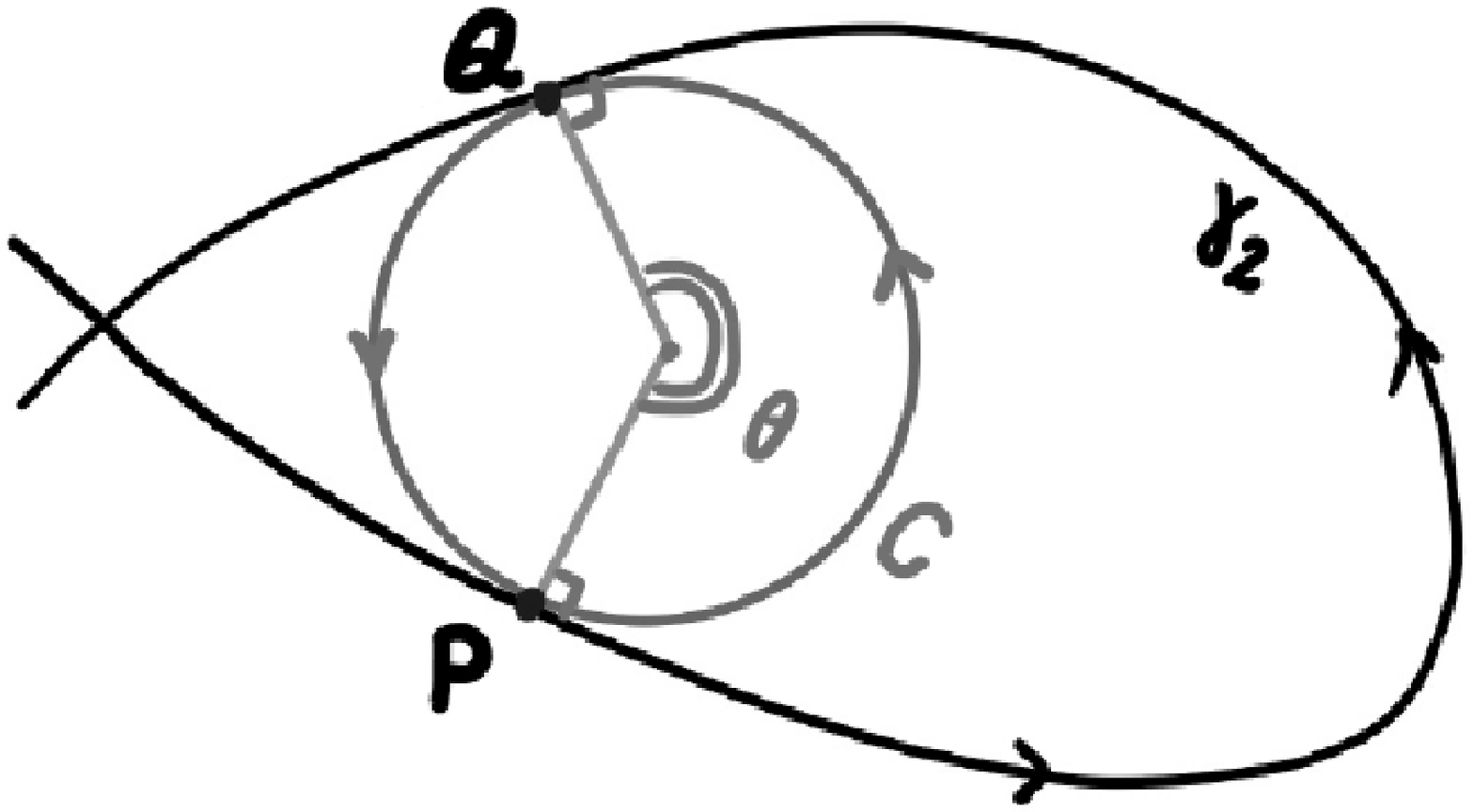}\quad
        \includegraphics[height=3.3cm]{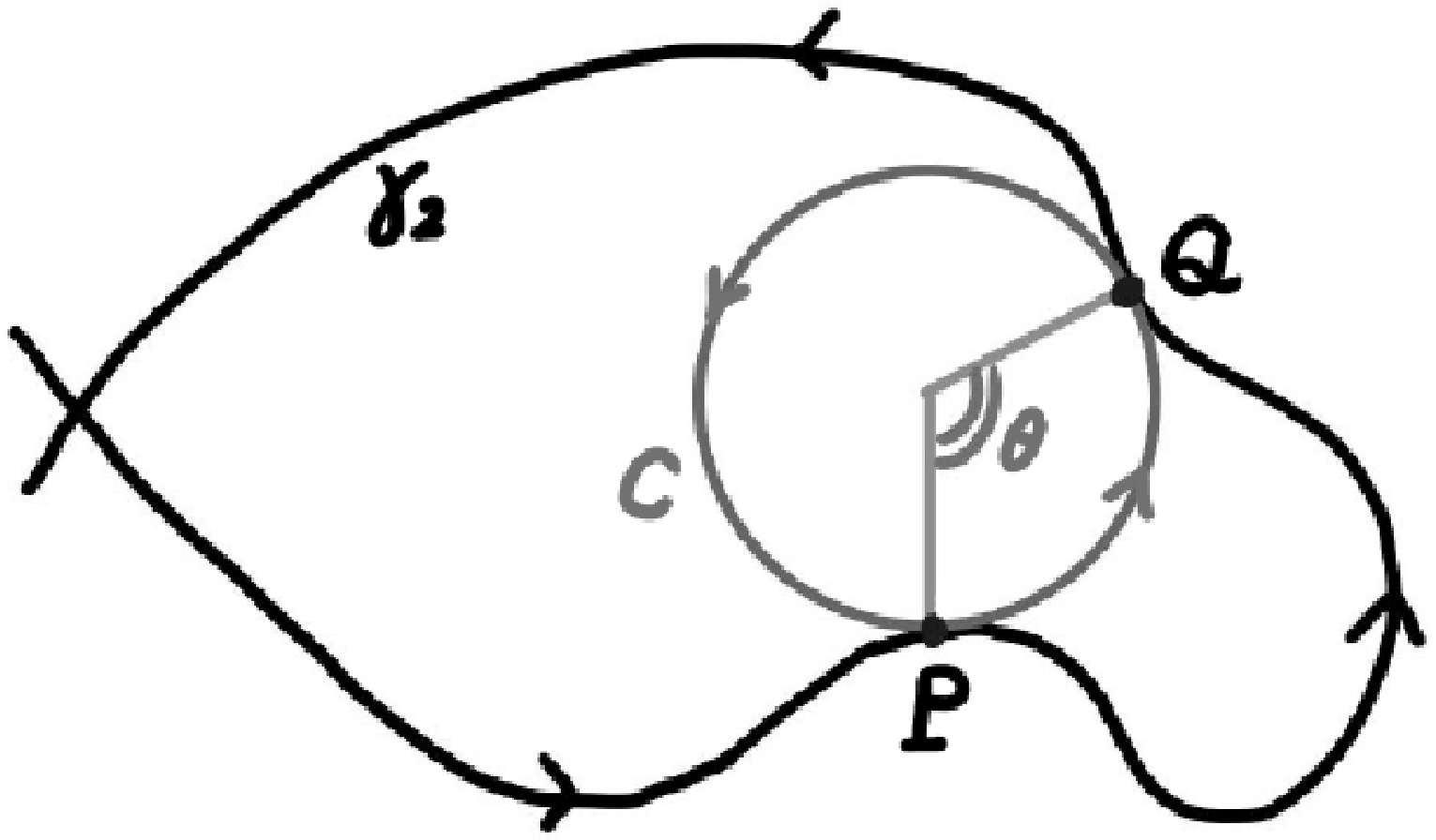}
\caption{
Inscribed $\mu$-circles of bi-tangent angle greater than $\pi$ (left)
and less than $\pi$ (right),
respectively }\label{Fig:im1a1b}
\end{center}
\end{figure}

\rm
For the sake of simplicity, we set
$C:=C_{t_1}$ and assume that 
the orientation of $C$ is
compatible with the orientation of the positive shell $\Sigma_2$. 
By definition, we have
$$
t_1=\inf\{t\in [a_2,b_2]\,;\, \gamma_2(t)\in C\}.
$$
We then set
$$
t_2:=\sup\{t\in [a_2,b_2]\,;\, \gamma_2(t)\in C\},
$$
and
$$
P:=\gamma_2(t_1),\qquad Q:=\gamma_2(t_2).
$$
By definition, it holds that
$
a_2<t_1<t_2<b_2.
$
We call the angle of the subarc of $C$ from $P$
to $Q$  the {\it bi-tangent angle}
of $C$. We denote it by $\theta$  
$($see Figure~\ref{Fig:im1a1b}$)$.
The following assertion plays an important role
in the next section:

\begin{proposition} \label{prop:keyM}
The bi-tangent angle $\theta$ of the $\mu$-circle $C$
is greater than or equal to $\pi$.
\end{proposition}

\begin{figure}[h]
\begin{center}
        \includegraphics[height=5.3cm]{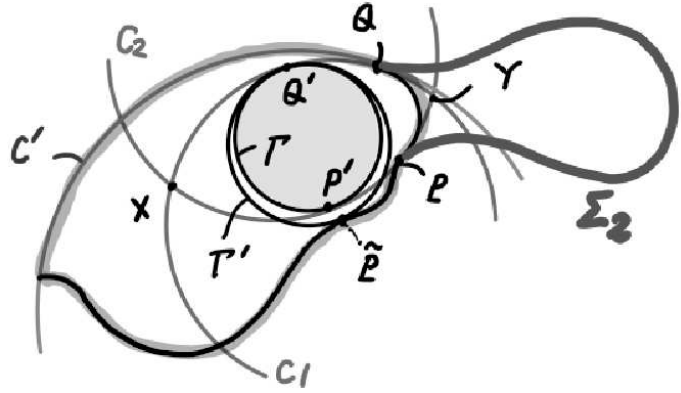}
\caption{
Proof of Proposition \ref{prop:keyM}
}
\label{Fig:keyfig}
\end{center}
\end{figure}

\begin{proof}
We prove the assertion by way of contradiction, so 
we suppose the bi-tangent angle of $C(:=C_{t_1})$ 
is less than $\pi$.
Then we define a $C^1$-shell joining
\begin{itemize}
\item the subarc of $\Sigma_2$ from the node $N_2$ of the
shell $\Sigma_2$ to $P$, 
\item the subarc of $C$ from $P$ to $Q$, and
\item the subarc of $\Sigma_2$ from $Q$ to
$N_2$.
\end{itemize}
We denote this $C^1$-differentiable shell by $\Sigma'_2$.
We let $\mu'$ be a number satisfying
$\mu>\mu'$.
Let $C_1$ (resp. $C_2$) be the $\mu'$-circle which is
tangent to $\Sigma_2$ at $Q$ (resp. at $P$).
We denote by $\Delta_{\Sigma'_2}$ the interior domain
bounded by $\Sigma'_2$.
Let $X,Y$ be the intersections of the two circles
$C_1$ and $C_2$.
If we take $\mu'$ sufficiently close to $\mu$,
then one of $\{X,Y\}$ lies in the
complement of the closure of the interior of $\Sigma'_2$
and the other lies in the interior of $\Sigma'_2$.
Without loss of generality,
we may assume $X$ lies in $\Delta_{\Sigma'_2}$.  
We then define a $C^1$-differentiable shell joining
\begin{itemize}
\item the subarc of $C_2$ from $X$  to $P$, 
\item the subarc of $C$ from $P$ to $Q$, and
\item the subarc of $C_1$ from $Q$ to
$X$.
\end{itemize}
We denote by $\Omega$ the interior domain of this 
newly defined $C^1$-differentiable shell. If we choose $\mu'$ sufficiently close to $\mu$,
then 
$
\Omega\subset \Delta_{\Sigma'_2}
$ holds.
We let $Q'$ (resp. $P'$) be
the midpoint of the arc bounded by $X$ and $Y$ on $C_1$ (resp. $C_2$).
Since the bi-tangent angle of $C$ is less than $\pi$,
we have $Q',P'\in \partial \Omega$, where $\partial \Omega$
is the boundary of the domain $\Omega$.
Moreover, there exists a unique circle $\Gamma$ which
is tangent to $\partial \Omega$
at two points $Q',P'$ satisfying
(cf. Figure \ref{Fig:keyfig})
\begin{itemize}
\item $\Gamma$ lies in $\overline{\Omega}$, and
\item the geodesic curvature of $\Gamma$ is less than $\mu$ 
and greater than $\mu'$.
\end{itemize}

We continuously expand $\Gamma$ keeping the property that it is tangent to
$C_1$ at $Q'$, and get a new circle $\Gamma'$ such that
\begin{itemize}
\item the geodesic curvature of $\Gamma'$ is less than that of $\Gamma$,
\item $\Gamma'$ lies in $\overline{\Delta_{\Sigma'_2} \cap \Delta_{C_1}}$,
\item $\Gamma'$ is tangent to $\gamma_2((a_2,t_1))$ at 
a certain point $\tilde P$.
\end{itemize}
We then consider the $\mu$-circle $\tilde C$ which is
inscribed in $\Gamma'$ and is tangent to $\Gamma'$ at $\tilde P$.
Since $\Gamma'$ is inscribed in $\Sigma'_2$, so is $\tilde C$.
Then $\tilde C$ must meet $\Sigma'_2$ only at $\tilde P$.
Since the image of $\Sigma'_2$ lies in the closure of the
interior domain of the positive shell $\Sigma_2$,
$\tilde C$ is a $\mu$-circle inscribed in $\Sigma_2$
which is tangent to  $\Sigma_2$ only at $\tilde P$.
This contradicts the definition of $t_1$.
\end{proof}

\section{The main theorem, its proof and applications}
\label{sec4}

\rm
We fix a $\mu$-admissible pair $(\gamma_1,\gamma_2)$.
The second and third authors  showed in \cite{KU}
that the extrinsic 
diameter of $f_{\gamma_1,\gamma_2}$ is $\pi$
under the assumption that
its mean curvature function does not change sign.
This assumption corresponds to the condition that
\begin{equation}\label{eq:KU}
\min_{t\in \R} \kappa_1(t)> \mu(>0),\qquad  \frac{-1}{\mu}> \max_{t\in \R} \kappa_2(t),
\end{equation}
which is stronger than the condition
\eqref{eq:k1k2}.
In this case, we can find a positive shell on $\gamma_1$
(cf. \cite[Prop. 3.7]{KU})
and can show the existence of an admissible bi-tangent of the 
second kind between $\gamma_2$ and
the positive shell on $\gamma_1$.
However, under our weaker assumption \eqref{eq:k1k2},
the argument given in \cite{KU} is not sufficient to show the following assertion:

\medskip
\noindent
{\bf Theorem B.}\label{thm:main}
{\it Let $(\gamma_1,\gamma_2)$ be a 
generic $\mu$-admissible pair.
Suppose that 
\begin{enumerate}
\item[(a)] either $\#(\gamma_1)$ or $\#(\gamma_2)$ is even, or
\item[(b)] 
each of $\#(\gamma_1),\#(\gamma_2)$ is odd, and
 $\gamma_1$ $($resp. $\gamma_2)$
has a negative $($resp. positive$)$
shell.
\end{enumerate}
Then the extrinsic diameter 
of $f_{\gamma_1,\gamma_2}$
is equal to $\pi$. }

\medskip
We prove this using a new idea different from the one
used in the proof of the main theorem in \cite{KU}.
Proposition A in the introduction follows from this assertion,
since a closed curve having only one crossing has
a positive shell and a negative shell at the same time. 

\begin{remark}
The main theorem of \cite{KU} is independent of Theorem B.
In fact, let $\gamma_1$ be a closed regular spherical
curve of positive geodesic curvature greater than $1$
whose topological type is $3_6$ as in Appendix B. 
Then $\gamma_1$ has 
three independent positive shells but has no 
negative shells. 
We let $\gamma_2$ be the closed curve obtained by reversing
the orientation of $\gamma_1$.
Then, by definition 
$$
\min_{t\in \R} \kappa_1(t)> 1,\qquad 
-1>\min_{t\in \R} \kappa_2(t).
$$
Since $\gamma_1$ has no negative shells,
$\gamma_2$ has no positive shells.
So the $\mu$-admissible pair ($\mu=1$)
$(\gamma_1,\gamma_2)$ does not satisfy
the assumption of Theorem B, but
satisfies \eqref{eq:KU}, and so
the extrinsic diameter of $f_{\gamma_1,\gamma_2}$
is $\pi$ by \cite[Proposition 4.7]{KU}.
\end{remark}

As pointed out in Fact \ref{fact:I1},
the case (a) is obvious.
So, to prove Theorem B, we may assume (b).
Then we can find a negative shell 
$\gamma_1|_{[a_1,b_1]}:[a_1,b_1]\to S^2$, where $[a_1,b_1]$ is
a subarc of $\R/l_1\Z$.
We set 
\begin{equation}\label{eq:Sigma1}
\Sigma_1:=\gamma_1([a_1,b_1]).
\end{equation}
To prove Theorem B, we prepare three lemmas:

\begin{lemma}\label{lem5-1}
There exists a $\mu$-circle $C$ such that
$\Sigma_1$ is inscribed in $C$
and meets $\Sigma_1$ at exactly two points.
\end{lemma}

We call such a $\mu$-circle a \em{bi-tangent $\mu$-circle 
of $\Sigma_1$}.

\begin{proof}
If we reverse the orientation of $\Sigma_1$,
then it becomes a positive shell.
So by the same argument as in
Lemma \ref{lem:Sigma2}, we can find
a bi-tangent $\mu$-circle $C$ of $\Sigma_1$.
It is sufficient to show that $C$ meet $\Sigma_1$ at
exactly two points. If not, there are three points
$P_1,P_2,P_3$ on $C$ in this cyclic order and
three values $t_1,t_2,t_3\in [a_1,b_1]$ 
$(t_1<t_2<t_3)$
such that
$$
P_j=\gamma_1(t_j) \qquad (j=1,2,3).
$$
By Fact \ref{fact:KU1}, $P_3$ (resp. $P_1$)
lies in  the past part (resp. the future part)
of $P_2$, and so the cyclic order of the
three points on $C$ must be $P_3,P_2,P_1$, a contradiction.
\end{proof}

\rm
Let $C$ be the $\mu$-circle given in Lemma \ref{lem5-1}.
We let $P,Q\in \Sigma_1$ 
be the two bi-tangent points of $C$
on $\Sigma_1$ such that
$$
P:=\gamma_1(d_1),\quad Q:=\gamma_1(c_1)
\qquad (a_1<c_1<d_1<b_1).
$$

\begin{figure}[h]
\begin{center}
\includegraphics[height=5.8cm]{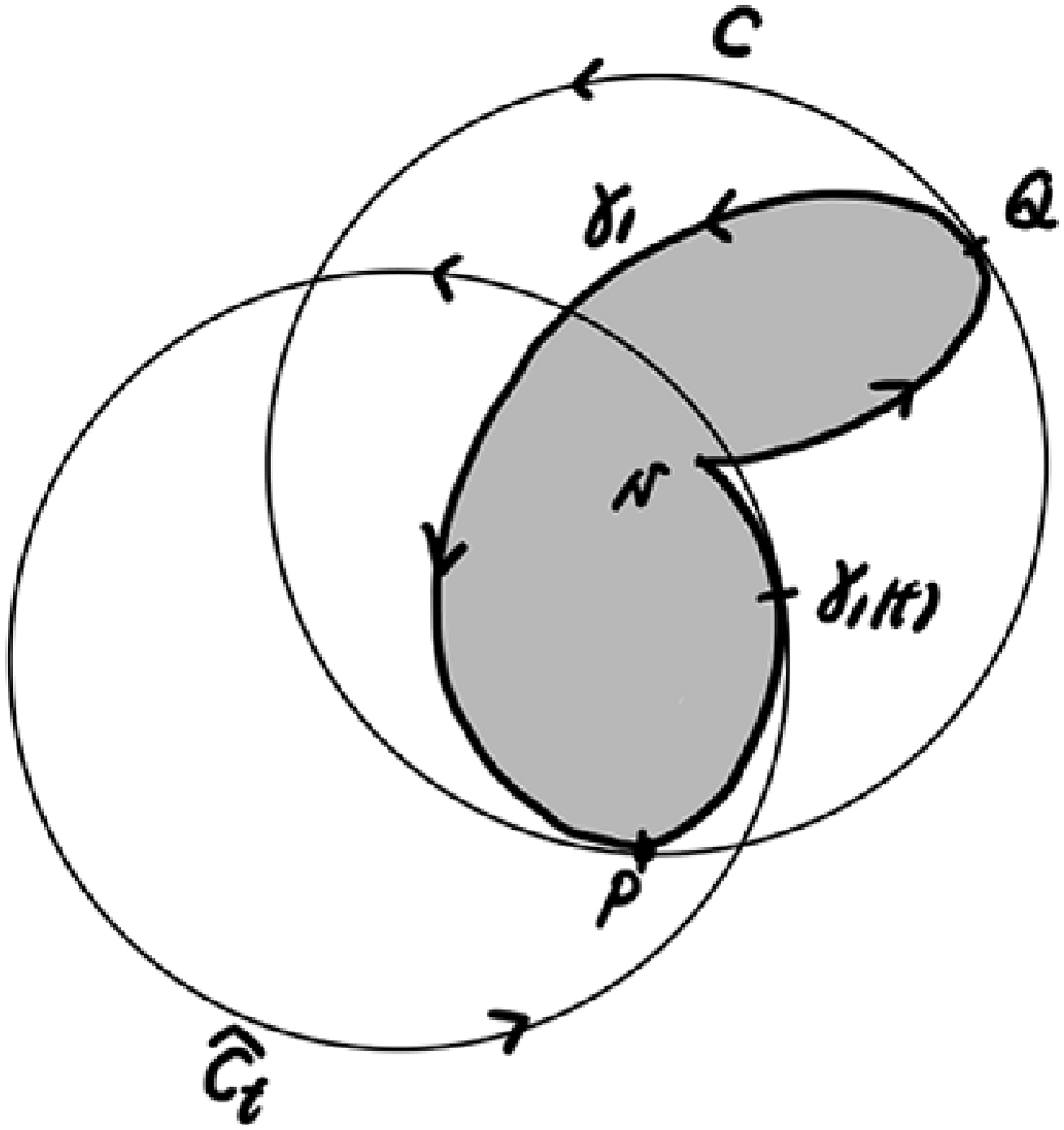}\qquad
\includegraphics[height=4.8cm]{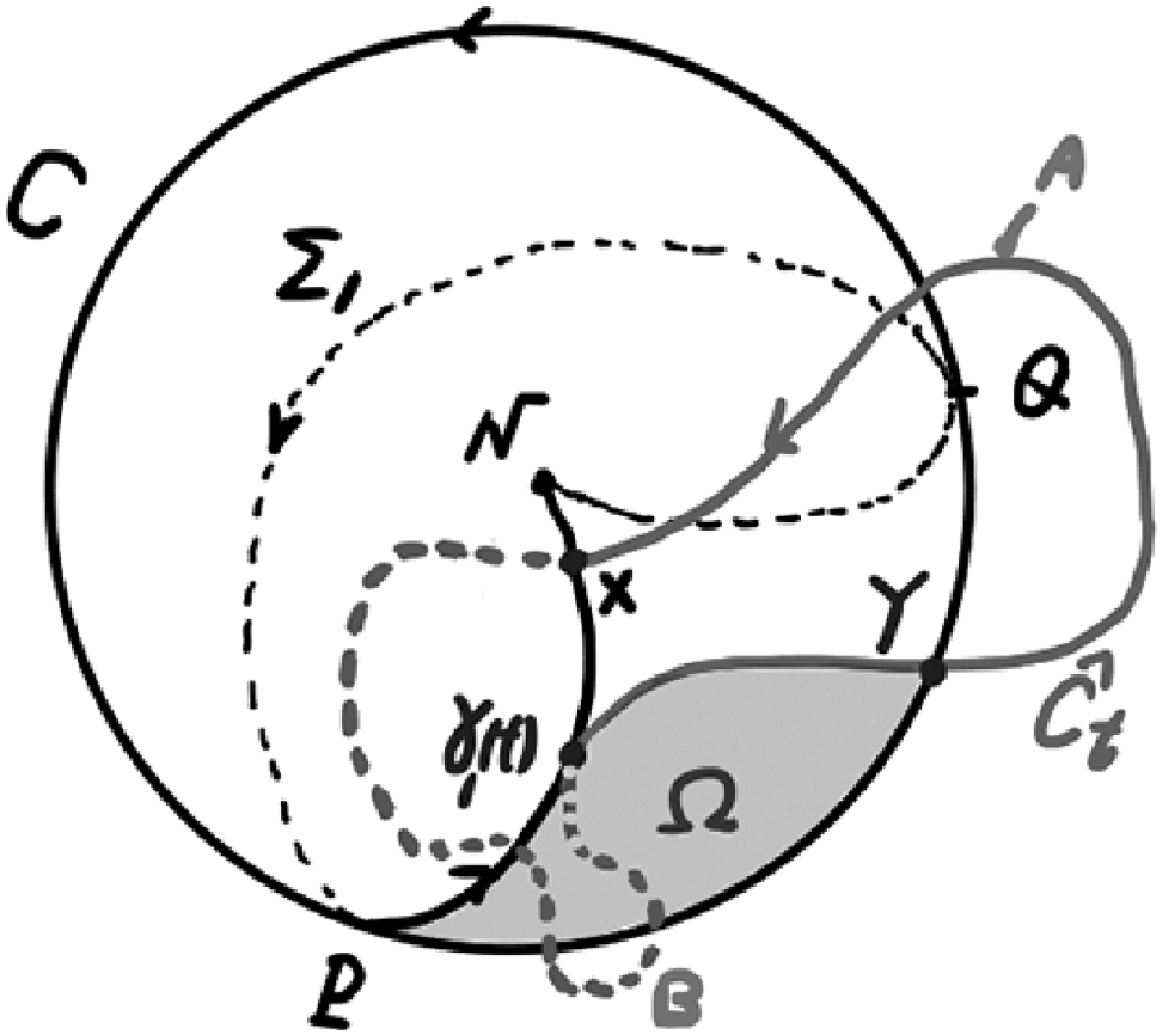}
\caption{
$\hat C_t$(left) and the proof of
Lemma \ref{lem5-3}
}
\label{Fig:final}
\end{center}
\end{figure}

For each $t\in [d_1,b_1]$,
we let $\hat C_t$ be the $\mu$-circle
which is tangent to $\Sigma_1$ at $\gamma_1(t)$
such that $\gamma_1(t')$ ($t'\ne t$)
lies in the interior domain of $\hat C_t$
whenever $t'$ is sufficiently close to $t$
(see Figure~\ref{Fig:final}, left).
Since $\gamma_1(t)$ lies in the interior domain
of $C$,  $\hat C_t$ meets $C$ at 
exactly two points on $S^2$. 

\begin{lemma}\label{lem5-2}
Suppse that there exists a point
$X=\gamma_1(x)$
$(x\in [d_1,b_1]\setminus \{t\})$
which lies on $\hat C_t$.
If $t<x\le b_1$ $($resp. $d_1\le x<t)$, 
then the subarc $\mathcal A$ of $\hat C_t$ 
from $\gamma_1(t)$ to $X$ 
(resp. from $X$ to $\gamma_1(t))$
meets the circle $C$ twice 
$($Figure \ref{Fig:final}, right in the case of $x>t)$.
\end{lemma}

\begin{proof}
If $x>t$ (resp. $x<t$),
then we set
\begin{align*}
&x_0:=\inf\{y\,;\, y\in (t,x]\,\mbox{ and }\gamma_1(y)\in \hat C_t\} \\
&\phantom{ssssssss}(\mbox{resp. }
x_0:=\sup\{y\,;\, y\in [x,t)\,;\,\gamma_1(y)\in \hat C_t\}
).
\end{align*}
Then $\gamma_1(x_0)\in \hat C_t$ holds.
We let $\mathcal A_0$
be the subarc of $\mathcal A$ between
$\gamma_1(t)$ and $\gamma_1(x_0)$.
Let $l$ be the length of the $\mu$-circle $\hat C_t $.
Suppose that $\mathcal A$ lies in $\overline{\Delta_C}$.
Since $C$ and $\hat C_t$ have common radius
and $\mathcal A\subset \overline{\Delta_C\cap \Delta_{\hat C_t}}$,
the length of $\mathcal A$ must be less than 
$l/2$.
Since the open subarc of $\gamma_1$ between $t$ and $x_0$
lies in the interior of $\hat C_t$, 
Fact \ref{fact:KU1} yields that
$\gamma_1(x_0)$ must lie in the past (resp. the future)
part of $\hat C_t$ with respect to $\gamma_1(t)$ 
if $x_0>t$ (resp. $x_0<t$).
Thus, the length of $\mathcal A_0$ must be greater than $l/2$.
This creates a contradiction, proving
the conclusion.
\end{proof}

\begin{lemma}\label{lem5-3}
For each $t\in [d_1,b_1]$,
$\hat C_t$ never meets $\gamma_1((t,b_1])$.
\end{lemma}

\begin{proof}
Suppose that 
$\hat C_t$ meets $\gamma_1((t,b_1])$ at $X=\gamma_1(x)$ ($t<x\le b_1$).
Let $\mathcal A$ be the subarc of $\hat C_t$ from $\gamma_1(t)$ to $X$.
By Lemma \ref{lem5-2}, $\mathcal A$ meets
$C$ twice. 
So we denote by $Y\in C$ 
the point where $\mathcal A$ starting from
$\gamma_1(t)$ meets $C$ the first time.
We then denote by $\mathcal A'$ the subarc of $\mathcal A$
from $\gamma_1(t)$ to $Y$.
Let $\mathcal B$ be the subarc of $\hat C_t$ from $X$
to $\gamma_1(t)$. Then 
$$
\mathcal A\cup \mathcal B=\hat C_t,\qquad  \mathcal A\cap \mathcal B
=\{X,\gamma_1(t)\}.
$$
Let $\mathcal E$ be the subarc of $C$
from $P$ to $Y$.
We let $\Omega(\subset \overline{\Delta_{C}})$ 
 be a domain 
bounded by the simple closed curve 
$\mathcal A'\cup \gamma_1([d_1,t])\cup \mathcal E$
(see Figure~\ref{Fig:final} right, where $\hat C_t$
is drawn so that it does not look like a circle).
Then any point on $\mathcal B$ sufficiently close to $\gamma_1(t)$
lies in $\Omega$.
Since $X\not \in \overline{\Omega}$, 
$\mathcal B$ must meet  
the boundary $\partial \Omega$.
However
\begin{itemize}
\item 
 $\mathcal B\setminus \{\gamma_1(t)\}$ cannot meet 
$\mathcal A'$, since $\hat C_t$ has no self-intersections.
\item
$\mathcal B$ cannot meet $C$,
since $\mathcal A$ meets $C$ twice and 
the radius of $\hat C_t$ and $C$ is common. 
\end{itemize}
So there must be a point $Z=\gamma_1(z)$ ($z\in (d_1,t)$)
such that $Z\ne \gamma_1(t)$ and $Z\in \mathcal B$.  
By  Lemma~\ref{lem5-2},
 $\mathcal B$ meets $C$ twice.
Then $\hat C_t(=\mathcal A\cup \mathcal B)$ meets $C$
four times, a contradiction. 
\end{proof}

\begin{figure}[htb]
\begin{center}
        \includegraphics[height=3.1cm]{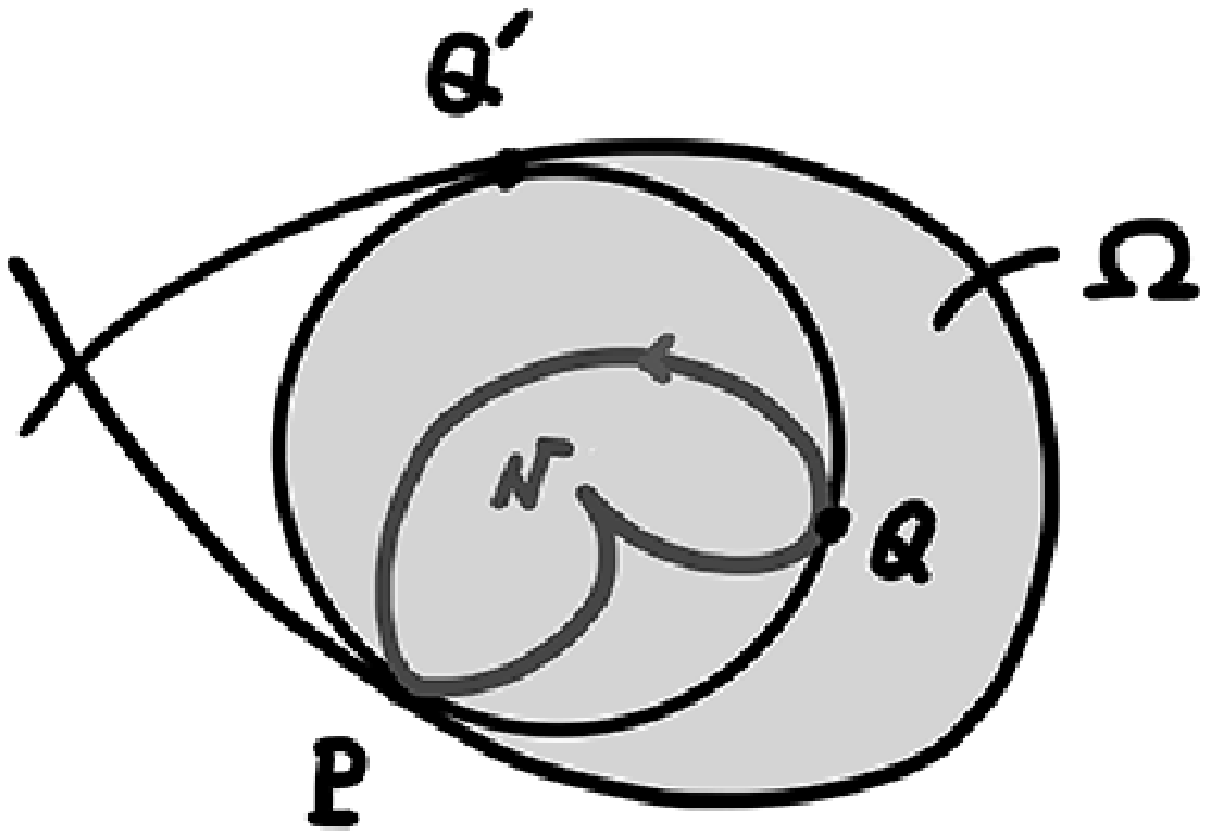}
        \includegraphics[height=3.2cm]{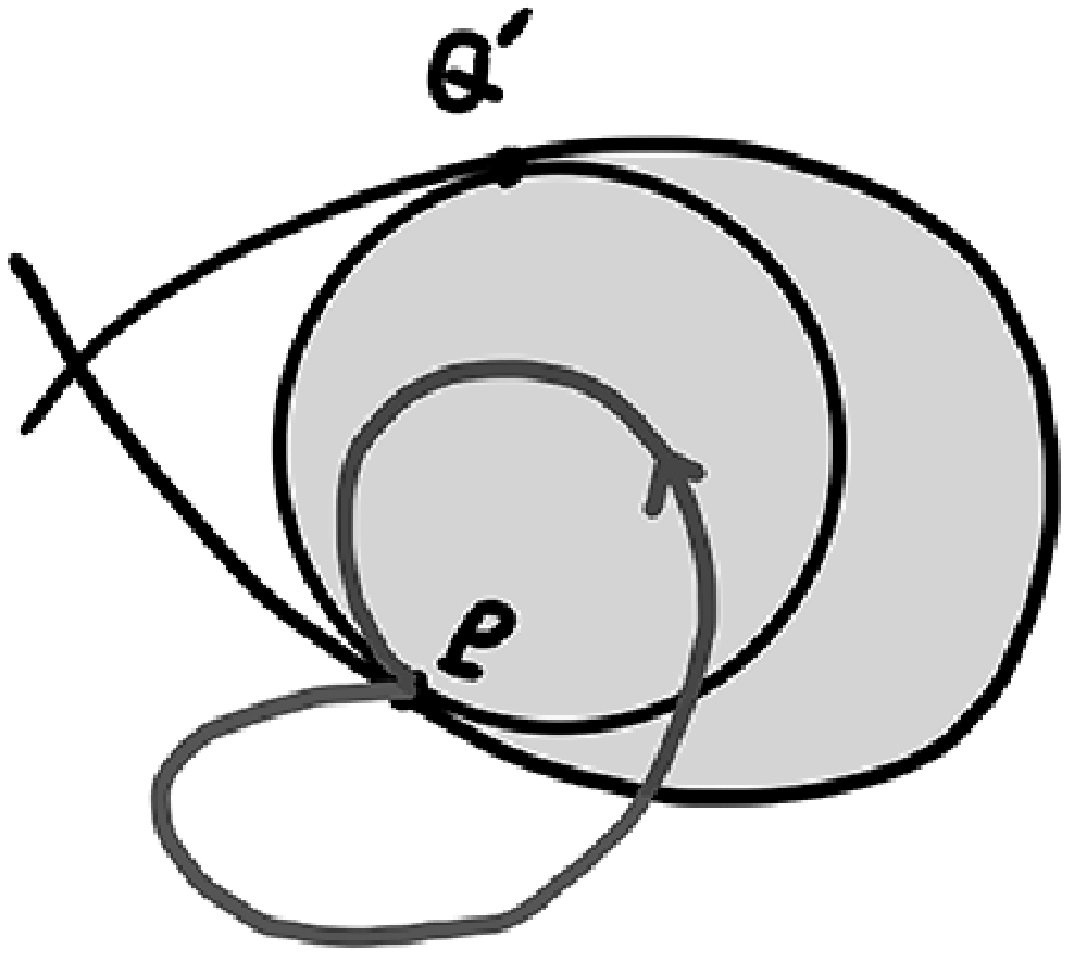}
        \includegraphics[height=3.2cm]{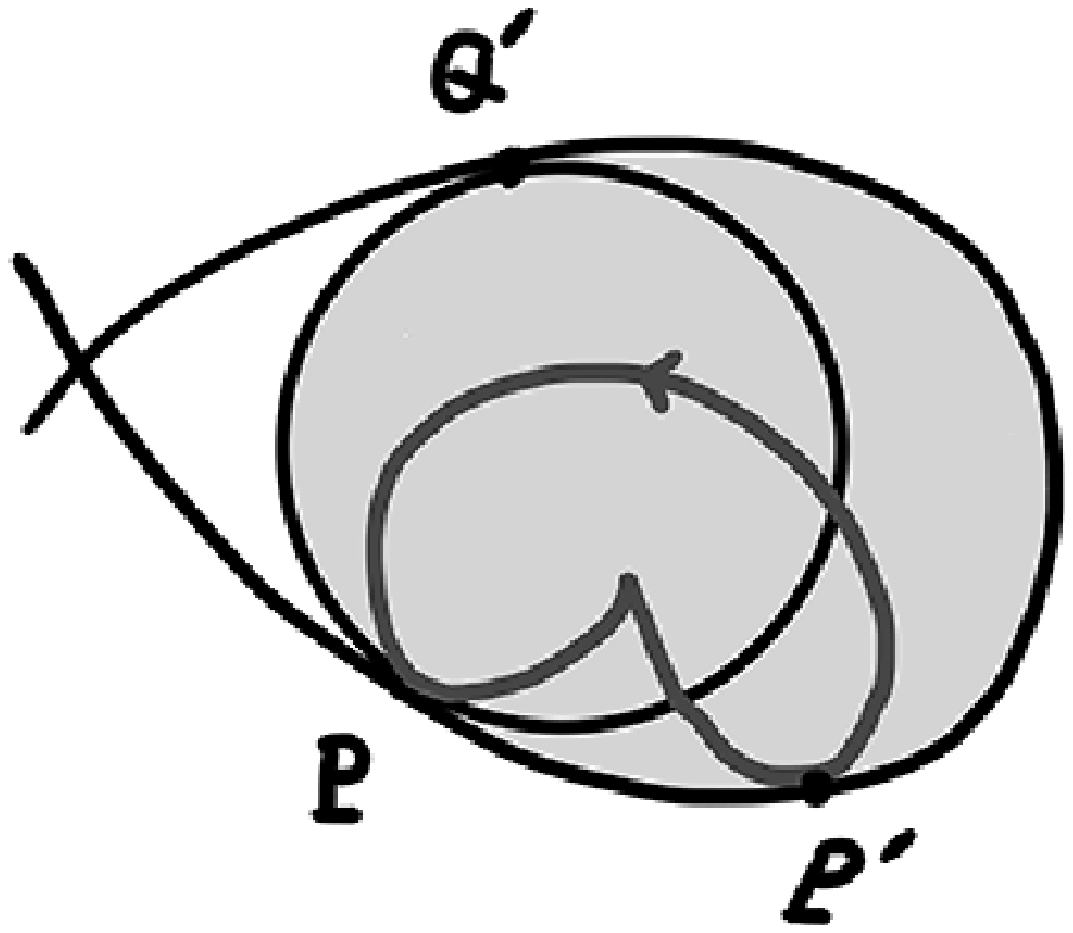}
\caption{Proof of Theorem B
}\label{Fig:bi-tangents2}
\end{center}
\end{figure}

\begin{proof}[Proof of Theorem B]
Let $\Sigma_2=\gamma_2([a_2,b_2])$ 
(cf.~\eqref{eq:Sigma2})
be the image of the positive shell given 
in Section \ref{sec3},
and let $C_{t_1}$ be the bi-tangent $\mu$-circle
inscribed in $\Sigma_2$.
To find the desired bi-tangent between $\gamma_1$
and $\gamma_2$, the initial position of $\gamma_1$ 
can be arbitrarily chosen.
So we may set the $\mu$-circle $C$ associated with
$\Sigma_1$ to coincide with 
$C_{t_1}$, and $\Sigma_1$ (cf. \eqref{eq:Sigma1})
is tangent to $\Sigma_2$ at the point $P$.
We set (cf. Figure \ref{Fig:bi-tangents2}, left)
$$
P:=\gamma_2(c_2)=\gamma_1(d_1),\quad
Q:=\gamma_1(c_1),\quad Q':=\gamma_2(d_2).
$$
Since 
\begin{itemize}
\item 
the bi-tangent angle of $C$ of 
$\Sigma_2$ is greater than or equal to
$\pi$ (cf.  Proposition  \ref{prop:keyM}) and
\item $P$ is the past part of $Q$ on $C$
(cf. Fact \ref{fact:KU1}),
\end{itemize}
\noindent
the three points
$P,Q,Q'$ lie on the circle $C(=C_{t_1})$
in this cyclic order.
We denote by $[PQ],[QQ']$ and $[Q'P]$ the subarcs of $C$
from $P$ to $Q$, $Q$ to $Q'$ and $Q'$ to $P$, respectively. 
Consider the $C^1$-differentiable simple closed curve given by
$
\mathcal S:=[Q'P]\cup \gamma_2([c_2,d_2]).
$
Let $\Omega$ be the open domain bounded by $\mathcal S$ which lies in the
closure of the interior domain of $\Sigma_2$.

For each $t\in [d_1,b_1]$,
there exists a unique orientation-preserving isometry 
$
\phi_t:S^2\to S^2
$
such that
\begin{enumerate}
\item $\phi_t(\gamma_1(t))=P$,
\item $\phi_t(\Sigma_1)$ is tangent to $\Sigma_2$ at $P$, and 
\item $\phi_t(\hat C_t)=C$.
\end{enumerate}
We then set
$$
\tau:=\sup(A),\qquad A:=\{t\in [d_1,b_1]\,;\, 
\mbox{$\phi_t(\Sigma_1)\subset \overline{\Omega}$}\}.
$$
Since $d_1\in A$,
the set $A$ is non-empty.
By \eqref{eq:k1k2}, we have $\tau>d_1$.
By the center of Figure~\ref{Fig:bi-tangents2},
$\tau<b_1$ holds.
 Suppose that $\phi_{\tau}(\Sigma_1)$ meets $[Q'P]$.
We know that $\phi_{\tau}\circ \gamma_1((\tau,b_1])$ 
lies in $\Delta_C$ (cf. Lemma~\ref{lem5-3}). 
In particular, it cannot meet $[Q'P]$.
Since the angle of subarc of $C$ from $Q'$ to $P$
is less than or equal to $\pi$
(cf. Proposition \ref{prop:keyM}),
$\phi_{\tau}\circ \gamma_1([a_1,\tau])$ also cannot 
meet $[Q'P]$, by Fact \ref{fact:KU1}. 
Since $\phi_{\tau}\circ \gamma_1((\tau,b_1])$ 
lies in $\Delta_C$, 
there exist $c'_1\in [a_1,\tau)$ 
and $c'_2\in (c_2,d_2))$
such that
$\phi_{\tau}(\Sigma_1)$ meets $\gamma_2((c_2,d_2))$ at 
the point 
$$
P'=\phi_{\tau}(\gamma_1(c'_1))=\gamma_2(c'_2).
$$
Then $(P,P')$ gives an admissible bi-tangent of the second kind
(cf.  Figure \ref{Fig:bi-tangents2}, right),
since this bi-tangent is exactly the case of
Figure 3, right.
\end{proof}

As a corollary, we can prove the following,
which contains Proposition A as a special case:

\begin{corollary}\label{cor:3}
Let $(\gamma_1,\gamma_2)$ be a $\mu$-admissible pair 
such that $\#(\gamma_i)\le 3$ $(i=1,2)$. Suppose that
\begin{itemize}
\item 
$\gamma_1$ is not of type $3_j$ $(j=3,6)$
$($ignoring their orientation$)$ as in 
Figure \ref{Fig:table1} of Appendix B, and
\item 
$\gamma_2$ is not of type $3_j$ $(j=3,6)$
as in  Figure \ref{Fig:table1}
including their orientation. 
\end{itemize}
Then $f_{\gamma_1,\gamma_2}$ has extrinsic diameter $\pi$.
\end{corollary}

\begin{proof}
If $\#(\gamma_1)$ or $\#(\gamma_2)$ is even, the conclusion is
obvious, and the case $\#(\gamma_1)=\#(\gamma_2)=1$
was the case of Proposition A.
So we need to consider only the case
$\#(\gamma_1)=3$ or $\#(\gamma_2)=3$.
If $\gamma_i$ ($i=1,2$)
is not of type $3_j$ ($j=3,6$),
then $\gamma_i$ has a positive shell and a negative shell
at the same time. So we can find a negative shell on $\gamma_1$
and  a positive shell on $\gamma_2$.

We suppose $\gamma_1$ is of type $3_3$ or $3_6$.
Since $\gamma_1$ has positive geodesic curvature,
$\gamma_1$ has a positive shell (cf. \cite[Prop. 3.7]{KU}), 
and in this case, $\gamma_1$ has no negative shells
and does not satisfy the assumption of Theorem B.
We next suppose that $\gamma_2$ is of type $3_3$ or $3_6$
with the orientation indicated in Figure~\ref{Fig:table1}. 
Then $\gamma_2$ does not admit positive shells
and does not satisfy the assumption of Theorem B.
(However, if $\gamma_2$ is of type $3_3$ or $3_6$
but the orientation is opposite, then $\gamma_2$ has 
a positive shell.)
So we get the conclusion.
\end{proof}

\rm
There are, in total, $10$ types of spherical curves
at most 3-crossings, and
only the two cases  $3_3,3_6$ are exceptional.
Similarly, we can prove the following:
 
\begin{corollary}\label{cor:3-6}
Let $(\gamma_1,\gamma_2)$ be a $\mu$-admissible pair 
such that $\#(\gamma_i)\le 5$ $(i=1,2)$.
Suppose that 
\begin{itemize}
\item $\gamma_1$ is not of type $3_j$ ($j=3,6$) as in 
Figure \ref{Fig:table1},
nor a 5-crossing curve as in one of the 17-curves as 
in Figure \ref{Fig:table2} $($ignoring their orientation$)$ of Appendix B, and
\item $\gamma_2$ is not of types $3_j$ ($j=3,6$),
nor one of  the 5-crossing curves given in the list of
Figure \ref{Fig:table2} of Appendix B including their orientation. 
\end{itemize}
Then $f_{\gamma_1,\gamma_2}$ has extrinsic diameter $\pi$.
\end{corollary}

There are, in total,  $105$ types of spherical curves with
at most 5-crossings, and the exceptional
case (i.e. curves that do not
admit a positive shell and a negative shell at the
same time) are $3_3,3_6$ and the $17$ curves listed
in Figure \ref{Fig:table2} of Appendix B.
So the cases for which Theorem B can be applied share
$$
74.5\%=\frac{\biggl((105-19)\times 2\biggr)\times 
\biggl((105-19)\times 2+19\biggr)}{(105\times 2)^2} \times 100.
$$ 
Since the number of curves with even crossings amongst 
$105$-curves is $22=1+2+19$.
So, if we do not apply Theorem B, we can judge only
$$
42.0\%=\frac{(22\times 2)\times (105 \times 2)+(105 \times 2)\times (22\times 2)
}{(105\times 2)^2} \times 100.
$$ 
So, at least for spherical curves within 5-crossings,
the percentage has extremely increased by Theorem B.

\appendix
\section{Properties of admissible pairs} 
\label{app:kitagawa}

In this appendix, we show that the assumption of
Theorem B does depend on the choice of
$\mu$-admissible pair which represents a common flat torus. 
For this purpose, we prepare a lemma
on parallel families of curves consisting of $\mu$-admissible
pairs.

Let $\g:\R\to S^2$ be a regular curve. 
For each $\theta \in \R$, the family of curves
$\g^{\theta}(s)$ ($\theta\in [0,2\pi)$)
in $S^2$ defined by 
\begin{equation}\label{eq10:adm}
\g^{\theta}(s) := (\cos\theta)\g(s) + (\sin\theta)n(s)
\qquad 
(s\in \R)
\end{equation}
is  called the {\it parallel curve} of the curve $\g(s)$.
It follows from (\ref{eq15:adm}) that 
\begin{equation}\label{eq20:adm}
(\g^{\theta})'(s) = (\cos\theta - \kappa(s)\sin\theta)\g{'}(s), 
\end{equation}
where $\kappa(s)$ denotes the geodesic curvature of $\g(s)$.

\begin{lemma}\label{lem10:adm}
Let $(\gamma_1,\gamma_2)$ be an admissible pair.
If the real number $\theta$ satisfies 
$$
|\cos\theta - \kappa_i(s)\sin\theta| > 0 
\qquad (i=1,2) 
$$
for all $s \in \R$,
then the curve $\g_i^{\theta}(s)$ is regular 
and its geodesic curvature 
$\kappa_i^{\theta}(s)$ satisfies the following:
\begin{itemize}
\item[(1)]
${\displaystyle \kappa_i^{\theta}(s) = \frac{\sin\theta + \kappa_i(s)\cos\theta}
{|\cos\theta -\kappa_i(s)\sin\theta|}}$,\\
\item[(2)]
$|(\g_i^{\theta})'(s)|^2(1 + \kappa_i^{\theta}(s)^2) =4$.
\end{itemize}
\end{lemma}

\proof
Since $|\cos\theta - \kappa(s)\sin\theta| > 0$, it follows 
from (\ref{eq20:adm}) that the curve $\g^{\theta}(s)$ is 
regular and satisfies the following relation: 
$$
\frac{(\g^{\theta})'(s)}{|(\g^{\theta})'(s)|} = 
\epsilon \frac{\g{'}(s)}{|\g{'}(s)|},\quad
\text{where}\quad 
\epsilon := \frac{
\cos\theta-\kappa(s)\sin\theta}{|\cos\theta-\kappa(s)\sin\theta|} \in \{1,-1\}.
$$
Hence, the unit normal vector field of $\g^{\theta}(s)$ is given by  
$$
n^{\theta}(s) = \g^{\theta}(s) \times \frac{(\g^{\theta})'(s)}{|(\g^{\theta})'(s)|} 
= \epsilon\biggl((\cos\theta)n(s) - (\sin\theta)\g(s)\biggr). 
$$
So, using (\ref{eq15:adm}), we obtain
\begin{equation}\label{eq30:adm}
(n^{\theta})'(s) = -\epsilon\biggl(\sin\theta + \kappa(s)
\cos\theta\biggr)\g{'}(s). 
\end{equation}
On the other hand, the Frenet formula implies
$
(n^{\theta})' = -\kappa^{\theta}(\g^{\theta})'.
$
Hence the assertion (1) follows from (\ref{eq20:adm}) and (\ref{eq30:adm}). 
Furthermore, the assertion (2) follows from (\ref{eq20:adm}) and (1).  
\qed\\

\begin{proposition}\label{Henkei:T20}
Let $(\G)$ be an admissible pair, 
and let $\theta$ be a real number such that the geodesic curvature of $\g_i$ satisfies 
$$
\cos\theta - \kappa_i(s)\sin\theta > 0 \qquad (i=1, 2).
$$
We set $\gb_i(s) := \g_i^{\theta}(s)$.
Then the pair $(\GB) $ is an admissible pair, and 
$f_{\GB}(s_1, s_2) = gf_{\G}(s_1, s_2)g^{-1}$ holds, 
where $g = \exp \left(\theta e_1/2\right)  \in S^3$.
\end{proposition}

\proof
It follows from Lemma \ref{lem10:adm} that 
$$
\bar{\kappa}_1(s_1) - \bar{\kappa}_2(s_2) 
= \frac{\kappa_1(s_1) - \kappa_2(s_2)}{(\cos\theta - \kappa_1(s_1)\sin\theta)
(\cos\theta - \kappa_2(s_2)\sin\theta)} > 0
$$
and 
$$
|\gb_i{'}(s)|^2(1 + \bar{\kappa}_i(s)^2) = |\g_i{'}(s)|^2(1 + \kappa_i(s)^2) = 4.
$$
So the pair $(\GB)$ is an admissible pair. 
Furthermore, it follows from the definition of $g \in S^3$ that 
\begin{equation}\label{eq40:adm}
Ad(g^{-1})e_3 = (\sin\theta)e_2 + (\cos\theta)e_3.
\end{equation}
Let $c_i(s)$ be the lift of $\widehat{\g_i}(s)$ with 
respect to the covering $p_2 : S^3 \to US^2$, 
and let $\cb_i(s) = c_i(s)g^{-1}$.
Since the unit normal vector field of $\g_i(s)$ is 
given by $n_i(s) = Ad(c_i(s))e_2$, 
it follows from (\ref{eq40:adm}) that
$$
Ad(\cb_i(s))e_3 = (\sin\theta)n_i(s) + (\cos\theta)\g_i(s) =  
\g_i^{\theta}(s) = \gb_i(s).
$$
Since $Ad(g^{-1})e_1 = e_1$ and $\cos\theta - \kappa_i(s)\sin\theta > 0$, 
\eqref{eq20:adm} implies
$$
Ad(\cb_i(s))e_1 = \frac{\g_i{'}(s)}{|\g_i{'}(s)|} =
\frac{(\g_i^{\theta})'(s)}{ |(\g_i^{\theta}){'}(s)|} 
= \frac{\gb_i{'}(s)}{|\gb_i{'}(s)|}.
$$
Therefore $p_2\circ \cb_i = \widehat{\gb_i}$, and so we obtain
$$
f_{\GB}(s_1, s_2) 
= \cb_1(0)^{-1}\cb_1(s_1)\cb_2(s_2)^{-1}\cb_2(0) 
= gc_1(0)^{-1}c_1(s_1)c_2(s_2)^{-1}c_2(0)g^{-1}.
$$
Hence $f_{\GB}(s_1, s_2) = gf_{\G}(s_1, s_2)g^{-1}$.
\qed

\begin{remark}
Let $(\G)$ be an admissible pair such that 
the geodesic curvatures $\kappa_1$ and $\kappa_2$ are bounded.
Define $\alpha\in (0,\pi)$ and $\beta \in (-\pi,0)$ by
$$
\cot \alpha := \sup_{s\in \R} \kappa_1(s),\quad
\cot \beta := \inf_{\alpha\in \R} \kappa_2.
$$
If $\beta < \theta < \alpha$, then $\cos\theta - \kappa_i(s)\sin\theta > 0$.
\end{remark}

\begin{proposition}
Let $(\G)$ and $(\GB)$ be admissible pairs. 
Suppose that 
$$
f_{\GB}(s_1, s_2) = gf_{\G}(s_1, s_2)g^{-1}
$$
for some $g \in S^3$. 
Then there exists a real number $\theta$ such that the geodesic 
curvature of $\g_i$ satisfies
$
\cos\theta - \kappa_i(s)\sin\theta > 0 \ (i=1,2),
$
and $(\GB) \equiv (\g_1^{\theta},\g_2^{\theta})$ holds.
\end{proposition}

\begin{proof}
Since $f_{\GB}=gf_{\G}g^{-1}$, we obtain
$\partial_if_{\GB}(0,0) = Ad(g)\partial_if_{\G}(0,0)$, and so
\begin{equation}\label{eq50:adm}
\partial_1f_{\GB}(0,0) \times \partial_2f_{\GB}(0,0) =
Ad(g)\biggl(\partial_1f_\G(0,0) \times \partial_2f_\G(0,0)
\biggr),
\end{equation}
where $\partial_i:=\partial/\partial s_i$ ($i=1,2$).

On the other hand, it follows from (\ref{eq30:pre}) and (\ref{eq3:adm}) that
$$
\partial_1f_\G(0,0) = \frac12|\g_1{'}(0)|\biggl(e_2 + \kappa_1(0)e_3\biggr), \quad
\partial_2f_\G(0,0) = -\frac12|\g_2{'}(0)|\biggl(e_2 + \kappa_2(0)e_3\biggr).
$$
Hence
$$
\partial_1f_\G(0,0) \times \partial_2f_\G(0,0) =
\frac14|\g_1{'}(0)|\,|\g_2{'}(0)|\biggl(\kappa_1(0) - \kappa_2(0)\biggr)e_1.
$$
Similarly, we obtain
$$
\partial_1f_{\GB}(0,0) \times \partial_2f_{\GB}(0,0) =
\frac14|\gb_1{'}(0)|\,|\gb_2{'}(0)|\biggl(\bar{\kappa}_1(0)-\bar{\kappa}_2(0)\biggr)e_1,
$$
where $\bar{\kappa}_i$ denotes the geodesic curvature of $\gb_i$.
So it follows from (\ref{eq50:adm}) that 
$
Ad(g)e_1 = e_1.
$
Then there exists a real number $\theta$ such that
\begin{equation}\label{eq60:adm}
		Ad(g)e_2 = (\cos\theta)e_2 + (\sin\theta)e_3, \quad
		Ad(g)e_3 = -(\sin\theta)e_2 + (\cos\theta)e_3.
\end{equation}
Let $c_i$ (resp. $\cb_i$) 
be a curve in $S^3$ such that $p_2\circ c_i = 
\widehat{\g_i}$ (resp. $p_2\circ \cb_i = \widehat{\gb_i}$).
Since $f_{\GB}(s,0)$ coincides with 
$gf_\G(s,0)g^{-1}$,  we obtain $\cb_1(0)^{-1}\cb_1(s) =
gc_1(0)^{-1}c_1(s)g^{-1}$. Hence
$$
\cb_1(s) = g_1c_1(s)g^{-1}, \quad
\text{where} \ \ g_1 = \cb_1(0)gc_1(0)^{-1}.
$$
Similarly, it follows from $f_{\GB}(0,s) 
= gf_\G(0,s)g^{-1}$ that
$$
\cb_2(s) = g_2c_2(s)g^{-1}, \quad
\text{where} \ \ g_2 = \cb_2(0)gc_2(0)^{-1}.
$$
Hence it follows from (\ref{eq60:adm}) that
$$
\gb_i(s) = Ad(\cb_i(s))e_3 
= Ad(g_i)Ad(c_i(s))\biggl((\sin\theta)e_2 + (\cos\theta)e_3
\biggr)
= Ad(g_i)\g_i^{\theta}(s).
$$
This shows that $(\GB) \equiv (\g_1^{\theta},\g_2^{\theta})$.
Furthermore, it follows from (\ref{eq20:adm}) that
$$
\gb_i{'}(s) = (\cos\theta-\kappa_i(s)\sin\theta)Ad(g_i)\g_i{'}(s).
$$
On the other hand, since
$$
\frac{\gb_i{'}(s)}{|\gb_i{'}(s)|} = Ad(\cb_i(s))e_1
= Ad(g_i)Ad(c_i(s))e_1 = \frac{Ad(g_i)\g_i{'}(s)}{|\g_i{'}(s)|}\qquad
(i=1,2),
$$
we have $\cos\theta-\kappa_i(s)\sin\theta = |\gb_i{'}(s)|/|\g_i{'}(s)| > 0$. 
\end{proof}

\begin{proposition}\label{deform:T30}
Let $(\G)$ be an admissible pair.
Suppose that there exists a real number $\theta \ (0 < \theta < \pi)$ satisfying 
\begin{equation}\label{eq70:adm}
\kappa_1(s_1) > \cot\theta > \kappa_2(s_2) 
\quad \text{for all} \ s_1, s_2 \in \R.
\end{equation}
We set $\gb_1(s):= \g_1^{\theta}(s)$ and $\gb_2(s):= \g_2^{\theta}(-s)$.
Then
\begin{itemize}
\item[(1)] 
$(\GB)$ is an admissible pair such that 
$\bar{\kappa}_1(s_1) > \cot\theta > \bar{\kappa}_2(s_2)$, 
\item[(2)]
there exists $a \in S^3$ such that 
$f_{\GB}(s_1, s_2) = af_{\G}(s_1, -s_2)a^{-1}$,
\item[(3)]
the mean curvature functions satisfy $H_{\GB}(s_1, s_2) =
 - H_{\G}(s_1, -s_2)$.
\end{itemize}
\end{proposition}

\proof
By the assumption (\ref{eq70:adm}), we obtain
$$
\cos\theta - \kappa_1(s)\sin\theta < 0, \qquad
\cos\theta - \kappa_2(s)\sin\theta > 0.
$$
So it follows from (\ref{eq20:adm}) that the curves 
$\g_i^{\theta}$ are regular curves, and

\begin{equation}\label{eq80:adm}
\frac{(\g_1^{\theta})'}{|(\g_1^{\theta})'| }
= - \frac{\g_1{'}}{|\g_1{'}|}, \qquad
\frac{(\g_2^{\theta})'}{|(\g_2^{\theta})'|} 
= \frac{\g_2{'}}{|\g_2{'}|}
\end{equation}
hold.
Since the geodesic curvatures of $\gb_1$ and $\gb_2$ are given by 
$$
\bar{\kappa}_1(s) = \kappa_1^{\theta}(s),\qquad
\bar{\kappa}_2(s) = -\kappa_2^{\theta}(-s),
$$
it follows from Lemma \ref{lem10:adm} that 
\begin{align*}
\bar{\kappa}_1(s) &= 
\frac{\sin\theta + 
\kappa_1(s)\cos\theta}{-\cos\theta + \kappa_1(s)\sin\theta} 
= \cot\theta + 
\left(\frac{1}{(\kappa_1(s) - \cot\theta)\sin^2\theta}\right) 
> \cot\theta,\\
\bar{\kappa}_2(s) &= -\frac{\sin\theta + \kappa_2(-s)\cos\theta}
{\cos\theta - \kappa_2(-s)\sin\theta} 
= \cot\theta + \left(\frac{1}{(\kappa_2(-s) - 
\cot\theta)\sin^2\theta}\right) < \cot\theta.
\end{align*}
This implies the assertion (1). 
Let $c_i$ be a curve in $S^3$ such that $p_2\circ c_i
=\widehat{\g_i}$. We set 
$$
\cb_1(s):=c_1(s)g^{-1}h, \quad \cb_2(s):=c_2(-s)g^{-1}h,
$$
where $g := \exp(\theta e_1/2)$ and $h$ is 
a point of $S^3$ such that 
$$
Ad(h)e_1 = - e_1, \quad Ad(h)e_2 = - e_2, \quad Ad(h)e_3 = e_3. 
$$
Since $Ad(g^{-1})e_3 = (\sin\theta)e_2 + (\cos\theta)e_3$ 
and $Ad(c_i(s))e_2 = n_i(s)$, we obtain
$$
Ad(\cb_1(s))e_3 = \g_1^{\theta}(s) = \gb_1(s),\quad
Ad(\cb_2(s))e_3 = \g_2^{\theta}(-s) = \gb_2(s).
$$
Since $Ad(g^{-1})e_1 = e_1$, it follows from (\ref{eq80:adm}) that
$$
	Ad(\cb_1(s))e_1 = -\frac{\g_1{'}(s)}{|\g_1{'}(s)|} 
	=\frac{(\g_1^\theta)'(s)}{|(\g_1^\theta)'(s)|} 
	= \frac{\gb_1{'}(s)}{|\gb_1{'}(s)|}. 
$$
Similarly, 
$
Ad(\cb_2(s))e_1 
= {\gb_2{'}(s)}/{|\gb_2{'}(s)|} 
$
holds.
This shows that $p_2\circ \cb_i = \widehat{\gb_i}$. Hence
$$
f_{\GB}(s_1, s_2) = \cb_1(0)^{-1}\cb_1(s_1)\cb_2(s_2)^{-1}\cb_2(0) = 
af_{\G}(s_1, -s_2)a^{-1}
$$
holds, where $a = h^{-1}g$.
This implies the assertion (2).
The assertion (3) follows from (2). 
\qed

The following assertion implies that 
the assumption of  Theorem B 
depends on the choice of 
$\mu$-admissible pairs
representing the same immersed torus, in general:

\begin{corollary}\label{prop:A}
There exist 
two $\mu$-admissible pairs $(\gamma_1,\gamma_2)$ 
and $(\bar \gamma_1,\bar \gamma_2)$
satisfying the following properties;
\begin{enumerate}
\item
$f_{\gamma_1,\gamma_2\equiv 
f_{\bar \gamma_1,\bar \gamma_2}}$,
\item 
$\gamma_1$ has no negative shells and
$\gamma_2$ has no positive shells, but 
\item 
$\bar \gamma_2$ has negative shells and
$\bar \gamma_1$ has positive shells. 
\end{enumerate}
\end{corollary}

\begin{proof}
For $a>0$, we set
$$
\sigma_a(t):=a (4 \cos t+3 \cos 2, -4 \sin t+3 \sin 2t),
$$
which is a closed regular curve in $\R^2$
with three crossings with 
positive curvature.
Consider a stereographic projection $p:S^2\to \R^2$
from the north pole, and set
$$
\tilde \sigma_a(t):=p^{-1}\circ \sigma_a(t)
\qquad (0\le t\le 2\pi) 
$$
that gives a closed  spherical curve of type $3_6$
in Figure \ref{Fig:table1} of Appendix B.
Obviously $\gamma^a$ has no negative shells.
If $a$ is sufficiently
 small,
$\tilde \sigma_a$ is positively curved.
Moreover, 
one can easily check that $n_a(t)=\gamma^{\pi/2}_a$
is  a closed curve of type $3_4$
with positive geodesic curvature,
and has three positive shells and three negative shells
at the same time. 
We set
$$
\gamma_1(t):=\tilde \sigma_a(t),\quad
\gamma_2(t):=\tilde \sigma_a(-t)\qquad (t\in \R),
$$
then $(\gamma_1,\gamma_2)$ is a $\mu$-admissible
pair for sufficiently small $\mu>0$  (for example $a=1/7$),
since $\gamma_1$ (resp. $\gamma_2$)
has positive (resp. negative) geodesic curvature
if $a$ is sufficiently small, for example $a\le 1/7$.
By definition, $\gamma_1$ has no negative shells,
and $\gamma_2$ has no positive shells.
We set $\mu:=\cot \theta$ ($\theta\in [0,\pi/2)$).
Then $\bar \gamma_1(t)$  and $\bar \gamma_2(t)$ are 
close to $n_a(t)$ and $n_a(-t)$ if $\mu$
is chosen to be sufficiently close to zero.
Since $n_a$ is of type $3_4$,
$\bar \gamma_1$  and $\bar \gamma_2$
have positive shells and negative shells at the same time.
\end{proof}

\begin{figure}[h!]
\begin{center}
        \includegraphics[width=6.9cm]{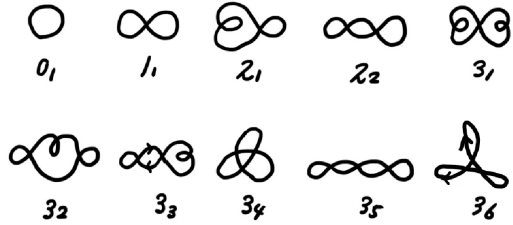}
\caption{Curves
with at most three crossings
}\label{Fig:table1}
\end{center}
\end{figure}

\section{
Tables of closed regular spherical curves
} 
\label{app:kob-U}

In \cite{KobU}, 
a table of generic closed spherical curves
with at most five crossings is given.
Figure \ref{Fig:table1} is the table within 3-crossings: 

\begin{figure}[h!]
\begin{center}
        \includegraphics[width=9.3cm]{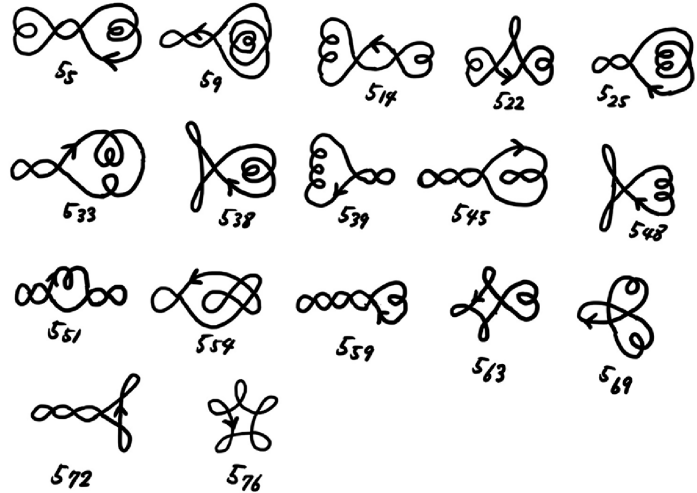}
\caption{Curves
with five crossings having only shells of the same kind
}\label{Fig:table2}
\end{center}
\end{figure}

There are in total $19$ (resp. $76$) 
curves with 4-crossings (resp. 5-crossings).
Figure \ref{Fig:table2}
 is the list of curves with 5-crossings
amongst them
whose shells are of the same kind:

\end{document}